\documentclass[12pt,reqno]{amsart}
\usepackage{amsmath,amsfonts,amssymb,amscd,amsthm,amsbsy,bbm,epsf,calc}
\usepackage{color}
\usepackage{datetime}

\usepackage{thmtools, thm-restate}
\usepackage{thm-patch, thm-kv, thm-autoref}
\usepackage[colorlinks=true, urlcolor=blue, linkcolor=blue, citecolor=black]{hyperref}

\numberwithin{equation}{section}
\newcounter{hours}\newcounter{minutes}

\theoremstyle{plain}
\declaretheorem[title=Theorem, parent=section]{thm}
\declaretheorem[title=Lemma,sibling=thm]{lem}
\declaretheorem[title=Corollary,sibling=thm]{cor}
\declaretheorem[title=Proposition,sibling=thm]{prop}
\declaretheorem[title=Definition,sibling=thm]{DEF}
\declaretheorem[title=Remark,sibling=thm]{rem}

\def\I{\mathcal I}

\def\M{{\mathcal M}}

\def\real{{\mathbb R}}

\def\ind{{\mathbbm{1}}}
\def\integer{{\mathbb Z}}

\def\ep{\varepsilon}
\def\al{\alpha}
\def\del{\delta}

\def\Om{\Omega}
\def\gam{\gamma} 
\def\lam{\lambda}
\def\Lam{\Lambda}

\def\grad{\nabla}

\def\Tr{\textnormal{Tr}}
\def\Id{\textnormal{Id}}

\def\intersect{\cap}

\DeclareMathOperator*{\osc}{osc}

\newcommand{\abs}[1]{\left| #1 \right|}

\newcommand{\ds}{\displaystyle}

\newcommand{\norm}[1]{\lVert#1\rVert}

\begin{document}

\title{Neumann Homogenization via Integro-Differential Operators  
}

\begin{abstract}
In this note we describe how the Neumann homogenization of fully nonlinear elliptic equations can be recast as the study of nonlocal (integro-differential) equations involving elliptic integro-differential operators on the boundary.  This is motivated by a new integro-differential representation for nonlinear operators with a comparison principle which we also introduce.  In the simple case that the original domain is an infinite strip with almost periodic Neumann data, this leads to an almost periodic homogenization problem involving a fully nonlinear integro-differential operator on the Neumann boundary.  This method gives a new proof-- which was left as an open question in the earlier work of Barles- Da Lio- Lions- Souganidis (2008)-- of the result obtained recently by Choi-Kim-Lee (2013), and we anticipate that it will generalize to other contexts.
\end{abstract}

\author{Nestor Guillen}
\author{Russell W. Schwab}

\address{Department of Mathematics\\
University of California, Los Angeles\\
520 Portola Plaza,
Los Angeles, CA  90095}
\email{nestor@math.ucla.edu}

\address{Department of Mathematics\\
Michigan State University\\
619 Red Cedar Road \\
East Lansing, MI 48824}
\email{rschwab@math.msu.edu}

\date{\today, arXiv version 3}

\thanks{The work of N. Guillen was partially supported by NSF DMS-1201413.  We wish to thank Inwon Kim for important remarks regarding \cite{ChoiKim-2013HomogNeumannArXiv}, \cite{ChKiLe-2012HomogNeumann}.  We also thank J\'er\'emy Firozaly and Cyril Imbert for pointing out an incorrect statement involving almost periodic functions in a previous version of this work, which led to a more direct proof of the main result.}


\maketitle
\baselineskip=14pt
\pagestyle{plain}              % page nos. at bottom, no headline
\pagestyle{headings}		% running headline and page nos. at top
\markboth{N. Guillen and R. Schwab}{Neumann Homogenization via Integro-Differential Operators}
%%%%%%%%%%%%%%%%%%%%%%%%%%%%%%%%%%%%%%%%%%%%%%

%%%%%%%%%%%%%%%%%%%%%%%%%%%%%%%%%%%%%%%%%%%%%%
%%%%%%%%%%%%%%%%%%%%%%%%%%%%%%%%%%%%%%%%%%%%%%
%%%%%%%%%%%%%%%%%%%%%%%%%%%%%%%%%%%%%%%%%%%%%%
%%%%%%%%%%%%%%%%%%%%%%%%%%%%%%%%%%%%%%%%%%%%%%
%%%%%%%%%%%%%%%%%%%%%%%%%%%%%%%%%%%%%%%%%%%%%%
\section{Introduction}\label{sec:Intro}
In this note, we introduce a method for Neumann homogenization (some background is given in Section \ref{sec:Background}) which exploits the Dirichlet-to-Neumann operator for fully nonlinear equations.  The strategy is to recast the original problem-- via the Dirichlet-to-Neumann operator-- as a global nonlocal homogenization problem posed on the boundary as well as to introduce an integro-differential inf-sup representation for the Dirichlet-to-Neumann map (Section \ref{sec:PerfectWorld}).  Using the inf-sup representation as motivation, this boundary homogenization becomes amenable to the more standard techniques which are a blend of the methods of \cite{Ishii-2000AlmostPeriodicHJHomog} and \cite{Schw-10Per} (although in a simplified setting).

The family of equations we study are fully nonlinear elliptic equations in an infinite strip with oscillatory Neumann data, given by

\begin{equation}\label{eqIntro:MainStripDomain}
	\begin{cases}
		\ds F(D^2u^\ep) = 0\ &\text{in}\ \Sigma^1\\
		u^\ep =0\ &\text{on}\ \Sigma_1\\
		\ds \partial_\nu u^\ep(x) = g(\frac{x}{\ep})\ &\text{on}\ \Sigma_0.
	\end{cases}
\end{equation}
The unit vector, $\nu$, gives the orthogonal to the domain, and
we use the following notation for the three pieces of the infinite strip:

\begin{align*}
	\Sigma^r &= \{ X\in\real^{d+1}\ :\ 0<X\cdot\nu<r\}\\
	\Sigma_r &= \{ X\in\real^{d+1}\ :\ X\cdot\nu=r\},\\
	\Sigma_0 &= \{ X\in\real^{d+1}\ :\ X\cdot\nu=0\}.
\end{align*}
We would like to point out that the extra Dirichlet condition, $u^\ep=0$ on $\Sigma_1$ in (\ref{eqIntro:MainStripDomain}), is artificial and only used to make existence/uniqueness a non-issue.  One could work in a half-space domain ($\{X\in\real^{d+1}\ :\ 0<X\cdot\nu\}$), but then must deal with various growth conditions on the set of admissible solutions, and we have chosen to avoid such issues by instead using the Dirichlet condition.

The main theorem says that oscillations introduced via $g$ on the boundary effectively average themselves out, and the solutions $u^\ep$ converge to an affine profile depending on $g$ and $\nu$. 

\begin{thm}\label{thm:HomogMain}
	If $g\in C^\gam(\Sigma_0)$ is almost periodic on $\Sigma_0$ (definition \ref{def:UniformApproxByPeriodic}) and $F$ is uniformly elliptic and positively 1-homogeneous (assumptions (\ref{eqSetUp:EllipticF}), (\ref{eqSetUp:1HomF})), then there exists a unique constant, $\bar g$, depending on $F$, $g$, and the choice to locate the auxiliary Dirichlet condition on $\Sigma_1$, such that $u^\ep\to\bar u$ and $\bar u$ is the unique solution of (\ref{eqIntro:MainStripDomain}) with Neumann condition $\partial_\nu \bar u(x)=\bar g$ on $\Sigma_0$.
\end{thm}

\begin{rem}
	In the case that $g:\real^{d+1}\to\real$ is $\integer^{d+1}$-periodic and that $\nu$ is an irrational direction, then $g|_{\Sigma_0}$ will be almost periodic.  We mention the $\nu$ irrational case because in the case that $\nu$ is rational, the original problem reduces to a half-space problem in which $g|_{\Sigma_0}$ is periodic with respect to some square lattice in $\Sigma_0$.  In this case, the result is already covered by \cite{BaDaLiSo-2008ErgodicProbHomogNeumannIUMJ}.  Furthermore, we note in the rational case that the effective Neumann condition will not be invariant by translations of the original domain, $\Sigma^1\mapsto y+\Sigma^1$ for some fixed $y$, whereas in the irrational case it is invariant.
\end{rem}

\begin{rem}
	We emphasize that the result stated in Theorem \ref{thm:HomogMain} is not new, and it appeared in \cite[Theorem 1.2(i)]{ChKiLe-2012HomogNeumann} as well as with a dependence on $F(D^2u^\ep,x/\ep)$ in \cite[Theorem 3.1]{ChoiKim-2013HomogNeumannArXiv}.  Instead we note the method employed here is different in both spirit and details to \cite{ChKiLe-2012HomogNeumann}, and we anticipate it will generalize to other settings both in homogenization as well as possibly other problems.  The case with $F(D^2u^\ep,x/\ep)$ and a more general domain were treated in \cite{ChoiKim-2013HomogNeumannArXiv}.
\end{rem}

Theorem \ref{thm:HomogMain} will be attacked via the Dirichlet-to-Neumann operator, and so we will need an auxiliary equation and some extra notation:

\begin{equation}\label{eqIntro:Dirichlet}
	\displaystyle
	\begin{cases}
 	\displaystyle	F(D^2U) = 0 \ &\text{in}\ \Sigma^1\\
    U =0\ &\text{on}\ \Sigma_1\\
	\displaystyle	U=v\ &\text{on}\ \Sigma_0.
	\end{cases}
\end{equation}
The Dirichlet-to-Neumann operator is then defined as
\begin{equation}\label{eqIntro:DirichletNeumann}
	v\mapsto \I^1(v,x) := \partial_\nu U(x)\ \ \text{such that}\ \ U\ \text{solves (\ref{eqIntro:Dirichlet})}.
\end{equation}

It turns out (and will be explained below in Sections \ref{sec:AssumeInfSup}, \ref{sec:ThmProofs}) that (\ref{eqIntro:MainStripDomain}) imposes a global nonlocal equation in $\Sigma_0$ for the function, $u^\ep|_{\Sigma_0}$, via (\ref{eqIntro:Dirichlet}).  This equation reads

\begin{equation}\label{eqIntro:NonlocalEpsilonEq}
	\I^1(u^\ep|_{\Sigma_0},x) = g(\frac{x}{\ep})\ \ \text{in}\ \ \Sigma_0.
\end{equation}
Equation (\ref{eqIntro:NonlocalEpsilonEq}) is a homogenization problem with a ``uniformly elliptic'' operator, $\I^1$, and an \emph{almost periodic} right hand side, $g$.  Hence we can connect Theorem \ref{thm:HomogMain} with the methods of \cite{Ishii-2000AlmostPeriodicHJHomog} and \cite{Schw-10Per}.
The auxiliary equation (\ref{eqIntro:NonlocalEpsilonEq}) is the heart of our approach-- also a novel feature to the analysis of nonlinear Neumann problems-- and we develop the sufficiency of homogenizing (\ref{eqIntro:NonlocalEpsilonEq}) for homogenizing (\ref{eqIntro:MainStripDomain}) in Section \ref{sec:ProofThmMain}.   The main component of Theorem \ref{thm:HomogMain} is

\begin{thm}\label{thm:HomogNonlocalBoundaryEq}
	There exists a unique constant, $\Bar\I(0)$, such that 
	\[
	\norm{u^\ep|_{\Sigma_0}-\Bar\I(0)}_{L^\infty(\Sigma_0)} \to 0\ \ \text{as}\ \ \ep\to0.
	\]
\end{thm}

\begin{rem}
	The choice of notation $\bar \I(0)$ seems strange, but is intentional.  It is meant to indicate that in more general situations, one expects to be required to resolve an effective nonlocal operator in $\Sigma_0$, which we would call $\Bar\I$.  In the very special context of (\ref{eqIntro:MainStripDomain}), it turns out that one only needs to understand $\Bar \I$ for constant functions, i.e. $\bar\I(0)$.  In general one expects $u^\ep|_{\Sigma_0}\to\bar u$ where $\bar u$ is the unique solution of $\Bar \I(\bar u)=0$ in $\Sigma_0$.
\end{rem}

\begin{rem}
	Throughout this note, we use viscosity solutions of equations such as e.g. (\ref{eqIntro:MainStripDomain}), (\ref{eqIntro:Dirichlet}).  We have collected various useful results in Appendix \ref{sec:ExAndUn}.  The existence and uniqueness of viscosity solutions for the equations with which we work can be found in \cite{IshiiLions-1990ViscositySolutions2ndOrder}.  A general introduction and background to viscosity solutions can be found in \cite{CrIsLi-92}.
\end{rem}

Before closing the introduction, let us motivate our approach in the simplest possible setting.  We assume that in (\ref{eqIntro:MainStripDomain}) and (\ref{eqIntro:Dirichlet}) $F=\Delta$ and $v\in C^{1,\gam}(\Sigma_0)$.  For the sake of presentation we ignore the Dirichlet condition on $\Sigma_1$ and assume that the phrase ``$U$ solves (\ref{eqIntro:Dirichlet})'' is interpreted as $U=P*v$, where $P$ is the Poisson kernel for the half-space with boundary $\Sigma_0$.  Then it is well known that the Dirichlet-to-Neumann operator in (\ref{eqIntro:DirichletNeumann}) becomes the $1/2$-Laplacian
\begin{equation*}
	v\mapsto \I(v) = -(-\Delta)^{1/2}v.
\end{equation*}
We note that under the assumptions that $g\in C^{\gam}(\Sigma_0)$, then $u^\ep\in C^{1\gam}(\Bar\Sigma^1)$, and thus if we set $u^\ep|_{\Sigma_0}$ as Dirichlet data in (\ref{eqIntro:Dirichlet}), then $U^\ep\equiv u^\ep$ and in this simple situation, the homogenization can be stated as a global problem on $\Sigma_0$
\begin{equation*}
	-(-\Delta)^{1/2}(u^\ep|_{\Sigma_0})(x)=g(\frac{x}{\ep})\ \ \text{in}\ \Sigma_0.
\end{equation*}
This is, by the almost periodicity assumed on $g$, an almost periodic homogenization problem on $\Sigma_0$.  

In the nonlinear setting, $\I$ will be a nonlinear operator with a comparison principle such that 
\begin{equation*}
	\I(u^\ep|_{\Sigma_0},x)=g(\frac{x}{\ep})
\end{equation*}
admits unique-- and by \cite{MiSi-2006NeumannRegularity}-- classical solutions for each $\ep>0$.  It turns out that this, plus the almost periodicity of $g$ is enough to modify and combine the techniques of \cite{Ishii-2000AlmostPeriodicHJHomog}, respectively \cite{Schw-10Per} for almost periodic Hamilton-Jacobi, respectively periodic integro-differential homogenization to prove the main result.

%%%%%%%%%%%%%%%%%%%%%%%%%%%%%%%%%%%%%%%%%%%%%%
%%%%%%%%%%%%%%%%%%%%%%%%%%%%%%%%%%%%%%%%%%%%%%
%%%%%%%%%%%%%%%%%%%%%%%%%%%%%%%%%%%%%%%%%%%%%%
%%%%%%%%%%%%%%%%%%%%%%%%%%%%%%%%%%%%%%%%%%%%%%
%%%%%%%%%%%%%%%%%%%%%%%%%%%%%%%%%%%%%%%%%%%%%%
\section{Some Background}\label{sec:Background}

Generally speaking, homogenization is the process of studying how oscillations in the coefficients of an equation such as (\ref{eqIntro:MainStripDomain})-- and in many more general situations-- cause effective behavior at a macroscopic scale, which can be thought of as a nonlinear averaging principle. Each type of equation is different, but in this case it is expected that the two phenomena of oscillations introduced at an $\ep$-scale via the Neumann data $g(x/\ep)$ and uniform-in-$\ep$ H\"older estimates arising from uniform ellipticity will combine to enforce an averaged behavior at the macroscopic scale for $\ep<<1$.  For a general background on homogenization, some standard references are: \cite{Babuska-1976HomogenizationInterface1SIAM}, \cite{BeLiPa-78}, \cite{EnSo-08}, \cite{Evan-92PerHomog}, \cite{Jiko-94}.  The study of how the effective-- or ``averaged''-- equation arises \emph{inside} of the domain is well developed by this point, and a good list references including many recent results can be found in \cite{EnSo-08}.

The situation for determining effective behavior arising from oscillations on \emph{the boundary} of the domain is a somewhat different story, and it is less developed than the study of oscillations in the interior. In the divergence setting, when a \emph{co-normal} boundary condition is enforced,
\begin{equation*}
	\left(n(x),A(\frac{x}{\ep})\grad u^\ep(x)\right) = g(\frac{x}{\ep})\ \ \text{on}\ \partial \Om
\end{equation*} 
(if the equation is posed in $\Om$), the situation is well understood thanks to the divergence structure of the boundary condition and can be found in \cite{BeLiPa-78}.  The non-divergence case is much different and less is known.  The first works involved some special assumptions which either directly or indirectly require the boundary of the domain to have a periodicity which is more or less a sub-lattice of the periodic lattice for the bulk equation (in our context, that would mean, e.g. $\nu$ is rational and $g$ is $\integer^{d+1}$ periodic).  Some of these results are in \cite{Arisawa-2003HomogNeumannAIHP}, \cite{BaDaLiSo-2008ErgodicProbHomogNeumannIUMJ}, \cite{Tanaka-1984HomogDiffBoundary}, and they treat cases in which both the equation in the domain as well as on the boundary have oscillatory coefficients.  The approach of \cite{BaDaLiSo-2008ErgodicProbHomogNeumannIUMJ} is to solve corrector equations on both the interior and boundary of the domain.  The difficulty is that the corrector from the interior arises in the expansion on the boundary, and so one basically must resolve a nonlinearly coupled system of corrector equations. 

An important question to address was how to prove homogenization in situations when the boundary does not share any periodicity with the equation in the interior of the domain.  Recent progress was made for situations where some of the periodicity assumptions on the boundary of the domain can be relaxed in \cite{ChKiLe-2012HomogNeumann}, but they still require a translation invariant operator inside the domain, and the domain must be strip-like as is $\Sigma^1$.  The important extension to more general domains with $x/\ep$ dependent coefficients in both $F$ and $g$ was recently obtained in \cite{ChoiKim-2013HomogNeumannArXiv}.  Because of the structure of the equation in those settings, they did not need to resolve a corrector equation on the boundary of the domain.

The question of whether or not Neumann homogenization can be approached via an almost periodic boundary corrector problem utilizing e.g. \cite{Ishii-2000AlmostPeriodicHJHomog} was already raised in \cite[Section 5]{BaDaLiSo-2008ErgodicProbHomogNeumannIUMJ}.  Some (unpublished) progress involving integro-differential equations for the homogenization of the Neumann problem was subsequently made by Lions and Souganidis for some special cases involving a family of linear equations \cite{Soug-2013Personal}.  The approach we develop for Theorems \ref{thm:HomogMain} and Theorem \ref{thm:HomogNonlocalBoundaryEq} lends an answer of how to use almost periodic techniques for the boundary equation, and hence-- we hope-- puts the homogenization of Neumann problems in better alignment with existing techniques.

%%%%%%%%%%%%%%%%%%%%%%%%%%%%%%%%%%%%%%%%%%%%%%
%%%%%%%%%%%%%%%%%%%%%%%%%%%%%%%%%%%%%%%%%%%%%%
%%%%%%%%%%%%%%%%%%%%%%%%%%%%%%%%%%%%%%%%%%%%%%
%%%%%%%%%%%%%%%%%%%%%%%%%%%%%%%%%%%%%%%%%%%%%%
%%%%%%%%%%%%%%%%%%%%%%%%%%%%%%%%%%%%%%%%%%%%%%
\section{The Setup}\label{sec:Setup}

\subsection{Assumptions}

We will make the following assumptions on $F$ and $g$

\begin{itemize}
	\item Uniform Ellipticity: $F$ is uniformly elliptic with respect to the Pucci extremal operators for some $\Lam\geq\lam>0$, i.e. for any $u,v\in C^2$,
\begin{equation}\label{eqSetUp:EllipticF}
	\M^-_{\lam,\Lam}(D^2(u-v)) \leq F(D^2u) - F(D^2v) \leq \M^+_{\lam,\Lam}(D^2(u-v)),
\end{equation}
and we remark that in the case that $F$ is linear, this reduces to the usual assumption of ellipticity.
	\item Pucci's Extremal Operators are defined, using $e_i = e_i(D^2u)$ to represent the eigenvalues of $D^2u$, as
\begin{align}
	\M^-_{\lam,\Lam}(D^2u(x)) := \lam\sum_{e_i\geq 0} e_i + \Lam\sum_{e_i<0} e_i = \inf_{\lam\Id\leq A \leq \Lam\Id}(\Tr(AD^2u(x))),\label{eqSetUp:PucciMin}\\
	\M^+_{\lam,\Lam}(D^2u(x)) := \Lam\sum_{e_i\geq 0} e_i + \lam\sum_{e_i<0} e_i = \sup_{\lam\Id\leq A \leq \Lam\Id}(\Tr(AD^2u(x))).\label{eqSetUp:PucciMax},
\end{align}
and we choose to subsequently drop the subscripts $\lam,\Lam$ for the remainder of the note.

	\item Positive 1-Homogeneity: 
\begin{equation}\label{eqSetUp:1HomF}
	\text{for all}\ \al\geq 0,\ F(\al D^2u)=\al F(D^2u).
\end{equation}

	\item H\"older Continuity: for some $\gam >0$, $g\in C^\gam(\Sigma_0)$.
	
	\item Almost Periodicity: $g$ is almost periodic on $\Sigma_0$.
	
	\item Notation: we will use the notation that 
		\begin{itemize}
			\item $X\in\real^{d+1}$ is written as
\[
X=(x,x_{d+1}),\ \ \text{for}\ x\in\Sigma_0\ \text{and}\ x_{d+1}\in span(\nu).
\]
			\item the space $C^2(\Om)$ is the functions, $f$, on $\Om$ with finite $\norm{f}_{L^\infty}$, $\norm{Df}_{L^\infty}$, and $\norm{D^2f}_{L^\infty}$. 
		\end{itemize}
		
		\item $\min(x,y)=x\wedge y$

\end{itemize}

We will work with almost periodic functions.

\begin{DEF}\label{def:UniformApproxByPeriodic}
 $f:\Sigma_0\to\real$ is \emph{almost periodic} if it can be uniformly approximated on $\Sigma_0$ by trigonometric polynomials.
\end{DEF}

\begin{rem}\cite[Proposition 1.2]{Shubin-1978AlmostPeriodic}\label{rem:AlmostPeriodic}
	Two other equivalent and classical formulations are that $f:\Sigma_0\to\real$ is \emph{almost periodic} if 
	
	\begin{itemize}
		\item[(i)] The set
		\begin{equation*}
			\left\{ f(\cdot+z)\ :\ z\in\Sigma_0   \right\}
		\end{equation*}
		is precompact in the space $L^\infty(\Sigma_0)$.
	\end{itemize}
	and
	\begin{itemize}
		\item[(ii)] For any $\delta>0$, the set of $\del$-almost periods of $f$, 
		\[
		E_\delta := \{\tau\in\Sigma_0\ :\ \sup_{x\in\Sigma_0}\abs{f(x+\tau)-f(x)}< \del\},
		\]
		satisfies the following property: there exists a compact set, $K\subset \Sigma_0$, such that 
		\[
		(z+K)\intersect E_\delta \not= \emptyset\ \textnormal{for all}\ z\in\Sigma_0.
		\]
	\end{itemize}

\end{rem}

\subsection{Preliminary Results}

We are working with viscosity solutions of equations such as (\ref{eqIntro:MainStripDomain}), (\ref{eqIntro:DirichletNeumann}), (\ref{eqSetUp:DirichletrScale}), and so we will collect various standard and well known facts about the existence and uniqueness of weak solutions in Appendix \ref{sec:ExAndUn}.  Since all of the equations we use here have unique viscosity solutions, we will keep a blanket reference to Appendix \ref{sec:ExAndUn} for these types of questions for the rest of the note, and we freely use ``viscosity solution'' interchangeably with ``solution''.

We define the Dirichlet to Neumann operators for $F$ in $\Sigma_0$, $\I^{r}: C^{1,\gam} \to C^0(\Sigma_0)$ by
\begin{equation}\label{eqSetUp:IrConstruction}
	\I^{r}(\phi,x) := \partial_\nu U^r_\phi(x),\;\;\phi\in C^{1,\gam}(\Sigma_0)
\end{equation}
where $U^r_\phi = U^r$ is the unique viscosity solution of

\begin{equation}\label{eqSetUp:DirichletrScale}
	\left \{ \begin{array}{rll}
		F(D^2U^r) & = 0 \ &\text{in}\ \Sigma^{r},\\
		U^r & = 0\ &\text{on}\ \Sigma_{r},\\	
		U^r & = \phi\ &\text{on}\ \Sigma_0.		
	\end{array}\right.
\end{equation}

    As we shall see below, the Dirichlet to Neumann maps for the standard extremal operators $\mathcal{M}^{\pm}$ will be of use. They are defined as follows, given $\phi:\Sigma_0 \to \mathbb{R}$,  define,
   \begin{equation}\label{eqSetUp:NonlocalExtremalDef}
    	 M^{r,\pm}(\phi,y) : = \partial_\nu U_\phi^{r,\pm},
    \end{equation}
    where $U^{r,\pm}_\phi = U^{r,\pm}: \Sigma^r \to \mathbb{R}$ is the unique viscosity solution of
	\begin{equation}\label{eqSetUp:ExtremalDirichletrScale}
		\left \{ \begin{array}{rll}
			\mathcal{M}^{\pm}(D^2U^{r,\pm}) & = 0 \ &\text{in}\ \Sigma^{r},\\
			U^{r,\pm} & = 0\ &\text{on}\ \Sigma_{r},\\	
			U^{r,\pm} & = \phi\ &\text{on}\ \Sigma_0.		
		\end{array}\right.
	\end{equation}

\begin{lem}\label{lemSetup:DtoNExtremals}
	Given $u,v \in C^{1,\gam}(\Sigma_0)$, $\I^r$ in (\ref{eqSetUp:IrConstruction}), and $M^{r,\pm}$ defined by (\ref{eqSetUp:NonlocalExtremalDef}), we have the pointwise inequalities for all $y\in\Sigma_0$
	\begin{equation*}
		M^{r,-}(u-v,y) \leq \I^r(u,y)-\I^r(v,y) \leq M^{r,+}(u-v,y) .
	\end{equation*}
\end{lem}

\begin{proof}
	We first remark that by Theorem \ref{thm:MilakisSilvestreDirichletRegularity}, the operators $M^{r,\pm}$ are classically defined for $u$ and $v$.
	Let us prove the upper bound. By definition $F(D^2U^r_u) = F(D^2U^r_v) = 0$ in $\Sigma^r$, and therefore the function $U := U^r_u-U^r_v$ solves-- via (\ref{eqSetUp:EllipticF})--
    \begin{equation*}
   	 \left \{ \begin{array}{rll}
   	 	  \mathcal{M}^+(D^2U) & \geq 0 & \text{ in } \Sigma^r,\\
                  U & = 0 & \text{ on } \Sigma_r,\\
				  U & = u-v & \text{ on } \Sigma_0.
   	 \end{array}\right.    	
    \end{equation*}
    Then, let $W = U^{r,+}_{u-v}$ be the unique solution of (\ref{eqSetUp:ExtremalDirichletrScale}).
    Thus, by the comparison principle, $U\leq W$ in $\Sigma^r$. Since $U\equiv W$ on $\Sigma_0$, it follows that
    \begin{equation*}
    	 \partial_\nu U \leq \partial_\nu W \text{ in } \Sigma_0.
    \end{equation*}
    But by construction, $\partial_\nu U = \I^r(u,\cdot)-\I^r(v,\cdot)$ and $\partial_\nu W = M^{r,+}(u-v,\cdot)$ and the first pointwise bound follows. The lower bound is proved by another comparison argument, using $\mathcal{M}^-$ instead of $\mathcal{M}^+$. This proves the lemma.
\end{proof}

\begin{lem}\label{lem:Iis1Homogeneous}
	$\I^r$ is positively $1$-homogeneous for all $r$, i.e. for all $\phi\in C^{1,\gam}$ and $c>0$
	\[
	\I^r(c\phi,y) = c\I^r(\phi,y).
	\]
\end{lem}

\begin{proof}
	This is an immediate consequence of the 1-homogeneity of $F$ (\ref{eqSetUp:1HomF}) combined with the uniqueness for (\ref{eqSetUp:DirichletrScale}).  Indeed, if we replace $\phi$ by $c\phi$ -- for $c>0$-- in (\ref{eqSetUp:DirichletrScale}), then we see that the new function $cU^r$ solves the same equation with Dirichlet data $c\phi$ on $\Sigma_0$.  Hence 
	\[
	\partial_\nu U^r_{c\phi} = \partial_\nu cU^r_\phi,
	\]
	which gives
	\[
	\I^r(c\phi,y) = c\I^r(\phi,y).
	\]
\end{proof}

\begin{lem}\label{lem:TranslationInvariant}
	$\I^r$ is translation invariant. Namely, given any smooth $\phi$,
    \begin{equation*}
         \I^{r}(\phi,x+y) = \I^{r}(\tau_y \phi,x),\;\;\forall\;x,y\in\Sigma_0.
    \end{equation*}
    Here, $\tau_y$ denotes the shift operator by $y \in \Sigma_0$,
    \begin{equation*}
         \tau_y \phi (x) := \phi(x+y).
    \end{equation*}
\end{lem}

\begin{proof}
    As the proof goes by a standard argument, we only give a sketch. It relies on the uniqueness of solutions for the Dirichlet problem \eqref{eqSetUp:DirichletrScale} and the fact that the operation $U \to F(D^2U)$ commutes with translations, and particularly, translations which are parallel to $\Sigma_0$. Therefore, the function $U(\cdot+y)$ solves the same Dirichlet problem as $U^r_{\tau_y u}$, thus by uniqueness $U^r_{\tau_y u} = U(\cdot+y)$ in $\Sigma^r$. Taking their normal derivatives on $\Sigma_0$, the lemma follows.
    
\end{proof}

\begin{lem}\label{lem:AlVAladdConstants}
	Let $r$ be fixed, $c$ a constant, and $\phi\in C^{1,\gam}(\Sigma_0)$, then
	\begin{equation*}
		\I^{r}(\phi+c,y)=\I^{r}(\phi,y)- r^{-1}c.
	\end{equation*}
\end{lem}

\begin{proof}
     Let $U^r_v: \Sigma^{r}\to\mathbb{R}$ be given by \eqref{eqSetUp:DirichletrScale} with $U^r_\phi = \phi$ on $\Sigma_0$. Now, the function
     \begin{equation*}
     	 \tilde U^r := U^r+c-cr^{-1} x_{d+1}
     \end{equation*}	
     Solves \eqref{eqSetUp:DirichletrScale} but with $\tilde U^r=\phi+c$ on $\Sigma_0$. Accordingly,
     \begin{equation*}
     	 \partial_\nu U^r = \I^{r}(\phi,y)\textnormal{ and } \partial_\nu \tilde U^r = \I^{r}(\phi+c,y).
     \end{equation*}
     It follows then that $\partial_\nu \tilde U^r = \partial_\nu U^r-r^{-1}c$, and the lemma is proved.
\end{proof}

\begin{lem}\label{lem:Scaling}
	If $\phi$ solves
	\begin{equation*}
		\I^1(\phi,x) = g(\frac{x}{\ep})\ \ \text{in}\ \Sigma_0 
	\end{equation*}
	and $w(y)= \ep^{-1} \phi(\ep y)$, then $w$ solves
	\begin{equation*}
		\I^{1/\ep}(w,y) = g(y).
	\end{equation*}
\end{lem}

\begin{proof}[Proof of Lemma \ref{lem:Scaling}]
	We let $U^1_v$ and $U^{1/\ep}_w$ be the solutions to (\ref{eqSetUp:DirichletrScale}) with data given by respectively $v$, $w$.  We define $W$ as
	\[
	W(y) = \ep^{-1} U^1_v(\ep y).
	\]
	The homogeneity of $F$-- (\ref{eqSetUp:1HomF})-- ensures that $W$ is in fact a solution of (\ref{eqSetUp:DirichletrScale}) in $\Sigma^{1/\ep}$ with data on $\Sigma_0$ given by $w$.  Hence by uniqueness of viscosity solutions of (\ref{eqSetUp:DirichletrScale}), we conclude $W\equiv U^{1/\ep}_w$.  Thus we have
	\[
	\I^{1/\ep}(w,y) = \partial_\nu U^{1/\ep}_w(y) = \partial_\nu W(y) = \partial_\nu U^1_v (\ep y) = \I^1(v,\ep y) = g(y).
	\]
\end{proof}

    The following auxiliary functions will be useful for localizing points of maxima and minima.  Let
    \begin{equation*}
         \phi_1(x):=\frac{|x|^2}{1+|x|^2},
    \end{equation*}
    and for $R>0$ we will consider the functions
    \begin{equation}\label{eqSetUp:Phi_Psi_def}
         \phi_R(x) := \phi(x/R).%,\;\;\psi_R(x) := \phi(x/R)-1.
    \end{equation}

    \begin{prop}\label{prop:AuxiliaryFunctionsAreNice}
        There is a positive constant $C=C(\lambda,\Lambda,d)$ such that
        \begin{equation*}
              \|M^{r,\pm}(\phi_R,\cdot)\|_{L^\infty(\Sigma_0)}\leq CR^{-1},\;\;\forall\;R,r \geq 1/2. 
        \end{equation*}
    \end{prop}

\begin{proof}
    We will only treat the case of an upper bound for $M^{r,+}$, and the other cases follow analogously.   Consider $U=U^{r,+}_{\phi_R}$ defined in \eqref{eqSetUp:ExtremalDirichletrScale}, and the function ($A$ to be specified)
    \begin{equation*}
         Q(X) := \phi_R(x)+AR^{-1} x_{d+1}(2-R^{-1}x_{d+1}),\;\;X=(x,x_{d+1}) \in \Sigma^{r\wedge R}.
    \end{equation*}
    For brevity, we are using $r\wedge R$ to denote $\min\{r,R\}$. Note that, if $x_{d+1}= r\wedge R$ then
    \begin{equation*}
         AR^{-1}x_{d+1}(2-R^{-1}x_{d+1}) = AR^{-1} (r\wedge R)  (2- (R^{-1} r)\wedge 1),
    \end{equation*}
    which equals $A$ if $r\wedge R = R$ and remains non-negative otherwise. Since $U\leq 1$  everywhere in $\Sigma^{r\wedge R}$ and $U=0$ on $\Sigma_{r\wedge R}$ when $r\wedge R = r$, it follows that when $A\geq 1$ we have
    \begin{equation*} 
         Q(X) \geq U(X) \text{ on } \Sigma_{r\wedge R}.
    \end{equation*}
    Moreover, $Q=U$ on $\Sigma_0$ (by construction), thus $Q\geq U$ on $\partial \Sigma^{r\wedge R}$. On the other hand, 
    \begin{equation*}
        D^2Q(X) = \left ( \begin{array}{cc}
           R^{-2}(D^2 \phi_1)(x/R) & 0\\
           0 & -2R^{-2}A
        \end{array} \right ),
    \end{equation*}
    which yields
    \begin{equation*}
       \mathcal{M}^+(D^2Q) \leq R^{-2}\left ( \Lambda d \|D^2\phi_1\|_{L^\infty(\Sigma_0)}-2 \lambda A\right )
    \end{equation*}
    Therefore, taking $A := \max\{1, d\Lambda \lambda^{-1} \|D^2\phi_1\|_{L^\infty(\Sigma_0)}\}$ we have
    \begin{equation*}
        \M^+(D^2Q) <0 \text{ in } \Sigma^{r\wedge R},\;\; Q\geq U \text{ on } \partial \Sigma^{r\wedge R},
    \end{equation*}
    so this function $Q$ is a classical supersolution in $\Sigma^{r\wedge R}$, and by the comparison principle, $Q \geq U$ in $\Sigma^{r\wedge R}$. Since $Q(x_0,0)= U(x_0,0)$ for any $x_0 \in \Sigma_0$, it follows that
    \begin{equation*}
          2R^{-1} A  = \partial_\nu Q(x_0) \geq \partial_\nu U(x_0) = M^{r,+}(\phi_R,x_0).
    \end{equation*}
    This gives the desired upper bound for $M^{r,+}(\phi_R,x_0)$. The respective lower bound for $M^{r,+}(\phi_R,x_0)$ and the bounds for $M^{r,-}(\phi_R,x_0)$ are obtained in an entirely analogous manner and we omit the details. This shows that 
    \begin{equation*}
        |M^{\pm,r}(\phi_R,x_0)|\;\leq CR^{-1},\;\;\forall\; x_0 \in \Sigma_0,
    \end{equation*}
    for some $C=C(\lambda,\Lambda,d)$.
\end{proof}

\begin{lem}[Comparison principle for smooth functions]\label{lem:ClassicalComparisonForIOperator}
	Let $u,v : \Sigma_0 \to \mathbb{R}$ be bounded functions such that $\I^{r}(u,\cdot)$ and $\I^{r}(v,\cdot)$ are classically defined and
	\begin{equation*}
		\I^{r}(u,x) \geq \I^{r}(v,x) \;\;\forall\;x \in \Sigma_0.
	\end{equation*}
    Then,
    \begin{equation*}
         u(x) \leq v(x)\;\;\forall\;x\in\Sigma_0.
    \end{equation*}
\end{lem}
    \begin{proof}
         Arguing by contradiction, suppose that for some $\delta>0$
         \begin{equation*}
              \sup \limits_{x\in \Sigma_0} u-v = \delta.
         \end{equation*}	 
         For every $R>0$, we consider the function
         \begin{equation*}
              h_{R}(x) = u(x)-v(x)-2\delta\phi_R(x),\;\;x \in \Sigma_0.
         \end{equation*}
         where $\phi_R$ is the auxiliary function defined in \eqref{eqSetUp:Phi_Psi_def}. For the purposes of the proof, we will need to select parameters $R_0$,$R_1$ and $R_2$. First, let $R_0$ be large enough so that
         \begin{equation*}
              \sup \limits_{x\in B_{R_0}}\left\{ u-v\right\} \geq \delta/2.         
         \end{equation*}
         The parameter $R_1$ will be specified at the end of the proof, but for now, we will only consider those $R_1$ large enough so that
         \begin{equation*}
               \sup \limits_{x \in B_{R_0}}\phi_{R_1}(x) \leq 1/8.
         \end{equation*} 
		 We point out that a subsequently larger choice of $R_1$ will have an effect on $R_2$, but the value of $R_2$ does not change the definition of (and hence equation for) the auxiliary function $h_{R_1}$.
         The parameter $R_2$, determined by $R_1$ and $R_0$, is the smallest one such that
         \begin{equation*}
              \inf \limits_{x\in B_{R_2}^c}\phi_{R_1}(x) \geq 3/4,\;\; R_2 \geq R_0.
         \end{equation*}
         Combining the inequalities above, we have that
         \begin{equation*}
             \sup \limits_{x\in B_{R_2}} h_{R_1} \geq \sup \limits_{x \in B_{R_0}} h_{R_1} \geq \sup \limits_{x \in B_{R_0}} \left \{ u-v\right \}-\delta/4\geq \delta/4>0,
         \end{equation*}
         and
         \begin{equation*}
              \sup \limits_{x\in B_{R_2}^c} h_{R_1} \leq \sup \limits_{x \in B_{R_2}^c} \{ u-v\}-3\delta/2 \leq 0.
         \end{equation*}
         Hence, $\sup \limits_{\Sigma_0} h_{R_1} = \sup \limits_{B_{R_2}} h_{R_1}$ and compactness yields that
         \begin{equation*}
             \exists\; x_0 \in \Sigma_0 \textnormal{ s.t. } \sup \limits_{\Sigma_0} h_{R_1} = h_{R_1}(x_0)
         \end{equation*}
         Thus,  $u$ lies below $v+2\delta \phi_{R_1}+h_{R_1}(x_0)$ in $\Sigma_0$ and it is touched by it at $x_0$. Therefore,
         \begin{align*}
             \I^{r}(u,x_0) & \leq \I^{r}(v+2\delta \phi_{R_1}+h_{R_1}(x_0),x_0),\\
                           & =\I^{r}(v+2\delta \phi_{R_1},x_0)-r^{-1} h_{R_1}(x_0),\\
                           & \leq \I^{r}(v+2\delta \phi_{R_1},x_0)-(4r)^{-1}\delta ,                           
         \end{align*}
         The last two lines being due to Lemma \ref{lem:AlVAladdConstants} and the fact that $h_{R_1}(x_0)\geq \delta/4$ by construction. On the other hand, Lemma \ref{lemSetup:DtoNExtremals} says that
         \begin{equation*}
              \I^{r}(v+2\delta\phi_{R_1},x_0) - \I^{r}(v,x_0)\leq M^{r,+}(2\delta \phi_{R_1},x_0) = 2\delta M^{r,+}(\phi_{R_1},x_0).
         \end{equation*}
         Next, by Proposition \ref{prop:AuxiliaryFunctionsAreNice},
         \begin{equation*}
              M^{r,+}(\phi_{R_1},x_0) \leq CR^{-1}_1,\;\;C=C(\lambda,\Lambda,d)R^{-1}_1,
         \end{equation*}
         as long as $r,R_1\geq 1$. Hence,
         \begin{equation*}
              \I^{r}(u,x_0) \leq  \I^{r}(v,x_0)+ CR^{-1}_1 - (4r)^{-1}\delta.
         \end{equation*}         
         Finally, since $\delta>0$, $R_1$ may be taken large enough so that $R_1^{-1} < r^{-1}\delta/(4C)$, thus
         \begin{equation*}
             \I^{r}(u,x_0) < \I^{r}(v,x_0),
         \end{equation*}
         which yields a contradiction, it follows that $u\leq v$ in $\Sigma_0$, as we wanted.
    \end{proof}

\begin{lem}\label{lem:GlobalEpLevelBounds} 
	If $w$ solves
	\[
	\I^{r}(w,x) = g(x)\ \ \text{in}\ \Sigma_0,
	\]
	then
	\begin{equation*}
	-r\norm{g}_{L^\infty} \leq w \leq r\norm{g}_{L^\infty}
	\end{equation*}
\end{lem}

\begin{proof}
	Lemma \ref{lem:AlVAladdConstants} plus $\I^r(0,\cdot)=0$ gives that
	\begin{equation*}
	     \I^r(c,\cdot) = -\frac{1}{r}c.
	\end{equation*}
	Therefore, 
    \begin{equation*}
	     \I^r(r\|g\|_{\infty},\cdot) \leq \I^r(w,\cdot)\leq \I^r(-r\|g\|_{\infty},\cdot).
	\end{equation*}
Then by the comparison of solutions, Lemma \ref{lem:ClassicalComparisonForIOperator}, we have
	\begin{equation*}
		 - r\|g\|_\infty \leq w \leq r\|g\|_\infty.
	\end{equation*}
\end{proof}

\begin{lem}\label{lem:IrComparison}
	Suppose that $r_2\geq r_1$ and that $u\geq 0$, then
	\begin{equation*}
	     \I^{r_2}(u,y)\geq \I^{r_1}(u,y)\ \ \forall\ y\in\Sigma_0.
    \end{equation*}
\end{lem}

\begin{proof}
	Note that $\Sigma^{r_1} \subset \Sigma^{r_2}$, so $U^{r_2}_u$ is defined in $\Sigma^{r_1}$. Since $U^{r_2}_u = u \geq 0$ on $\Sigma_0$, the comparison principle implies that $U^{r_2}_u\geq 0$ in $\Sigma^{r_2}$, and in particular $U^{r_2}_u \geq 0$ on $\Sigma_{r_1}$. Moreover, $U^{r_1}_u$ and $U^{r_2}_u$ agree in $\Sigma_0$ and solve the same equation in $\Sigma^{r_1}$. Thus $U^{r_2}_u$ is a supersolution for the problem solved by $U^{r_1}_u$, so that $U^{r_1}_u\leq U^{r_2}_u$ everywhere in $\Sigma^{r_1}$.
	
	Since the two functions agree on $\Sigma_0$, their normal derivatives must be ordered, namely
	\begin{equation*}
		 \partial_\nu U^{r_2}_u(y) \geq \partial_\nu U^{r_1}_u(y)\ \ \forall\ y\in \Sigma_0,
	\end{equation*}
	and the lemma follows.
	
	\end{proof}

\begin{lem}\label{lem:RHSComparison}
	Let $r\geq 1$ be fixed. Suppose that there exist bounded, classical respectively sub and super solutions $w_1$ and $w_2$ to
	\begin{align*}
		\I^r(w_1,y) \geq c_1 + g(y)\ \ \text{and}\ \ \I^r(w_2,y) \leq c_2 + g(y)\ \  \text{in}\ \Sigma_0.
	\end{align*}
	Then $\displaystyle c_1-c_2\leq \frac{1}{r}\sup_{\Sigma_0}|w_1-w_2|$.
\end{lem}

\begin{proof}
    Let $\tilde w_2 := w_2 -r(c_1-c_2)$, then by Lemma \ref{lem:AlVAladdConstants} we have
    \begin{align*}
        \I^r(\tilde w_2,y) & = \I^r(w_2,y)+c_1-c_2,\\
                           & \leq g(y) +c_1,\\
                           & \leq \I^r(w_1,y).
    \end{align*}
    Then Lemma \ref{lem:ClassicalComparisonForIOperator} yields that $w_1\leq \tilde w_2$, or  $w_1 \leq \tilde w_2 = w_2-r(c_1-c_2)$. Rearranging,
    \begin{equation*}
          r(c_1-c_2) \leq w_2-w_1 \leq \sup \limits_{\Sigma_0}|w_1-w_2|
    \end{equation*}
    dividing by $r$ the lemma follows.
\end{proof}

%%%%%%%%%%%%%%%%%%%%%%%%%%%%%%%%%%%%%%%%%%%%%%
%%%%%%%%%%%%%%%%%%%%%%%%%%%%%%%%%%%%%%%%%%%%%%
%%%%%%%%%%%%%%%%%%%%%%%%%%%%%%%%%%%%%%%%%%%%%%
%%%%%%%%%%%%%%%%%%%%%%%%%%%%%%%%%%%%%%%%%%%%%%
%%%%%%%%%%%%%%%%%%%%%%%%%%%%%%%%%%%%%%%%%%%%%%
\section{The Proof In a Perfect World}\label{sec:PerfectWorld}

In this section we develop an integro-differential line of attack for (\ref{eqIntro:NonlocalEpsilonEq}).  If one knew a priori that $\I^1$ were an integro-differential operator satisfying certain assumptions similar to those in \cite{Schw-10Per}, then the homogenization strategy for integro-differential equations could be applied to (\ref{eqIntro:NonlocalEpsilonEq}) without too much modification.  It turns out that this will indeed be the case, however, the simple set-up of (\ref{eqIntro:MainStripDomain})-- namely translation invariance of $F$-- allows for a proof which does not invoke \cite{Schw-10Per} but is motivated by it.  The obstacle to carrying out this line of attack is proving some fine properties of the L\'evy measures appearing in an inf-sup representation for $\I^1$ (Section \ref{sec:NonlinearCourrege}).  Such properties of the L\'evy measures representing $\I^1$ are fundamental to the application of regularity theory for integro-differential operators, and it is not known in exactly which class the operators may be.  Hence we do not know which, if any, of the results \cite{CaSi-09RegularityIntegroDiff}, \cite{Chan-2012NonlocalDriftArxiv}, \cite{ChDa-2012NonsymKernels}, \cite{GuSc-12ABParma}, \cite{KassRangSchwa-2013RegularityDirectionalINDIANA}, \cite{Silv-2006Holder} may be applicable to the operators $\I^r$.  We mention these issues again below. 

We begin with an observation that in (\ref{eqIntro:Dirichlet}) if the operator, $F$, were linear then $\I^1$ defined in (\ref{eqIntro:DirichletNeumann}) would again be linear.  Furthermore $\I^1$ always satisfies a comparison principle (equivalent to a global maximum principle in the linear case) due to the fact that (\ref{eqIntro:Dirichlet}) also has a comparison principle between sub and super solutions in $\Sigma^1$. Thus, it is well known in the linear case (\cite[Theorem 1.5]{Courrege-1965formePrincipeMaximum}, also Theorem \ref{thm:LinearCourrege} below) that $\I^1$ must admit an integro-differential representation.  We use this as motivation to obtain a similar representation in the nonlinear case (producing an inf-sup of linear operators) which brings the equation exemplified by (\ref{eqIntro:NonlocalEpsilonEq}) into much closer alignment with the homogenization of nonlocal operators studied in \cite{Schw-10Per}, where an inf-sup structure was assumed.  Here we give brief overview of some of the details, and we expect to develop these ideas further in a subsequent work.  Some examples of a similar representation for \emph{local} operators with a comparison principle in the context of semigroups and viscosity solutions can be found in \cite{AlvarezLionsGuichardMorel-1993AxiomsImageProARMA}, \cite{BarlesSouganidis-1998NewApproachFrontsARMA}, and \cite{Biton-2001NonlinMonoSemiViscSolAIHP}.

\subsection{Courrege's Theorem and an Inf-Sup Representation}\label{sec:NonlinearCourrege}

We will use the space $C^2(\Om)$ to be the collection of functions, $f$, with continuous second derivatives on $\Om$ with $\norm{f}_{L^\infty}$, $\norm{Df}_{L^\infty}$, and $\norm{D^2f}_{L^\infty}$ all finite.  This next definition can be thought of as a nonlinear version of the more commonly known global non-negative maximum principle for linear operators.

\begin{DEF}\label{def: operators with maximum principle}
	 A map $I:C^2(\mathbb{R}^d) \to C(\mathbb{R}^d)$ is said to satisfy the global comparison principle if given $u,v \in C^2(\mathbb{R}^d)$ such that $u$ touches $v$ from below at $x_0$ then,
     \begin{equation*}
          I(u,x_0)\leq I(v,x_0).	
     \end{equation*} 
     Here, ``$u$ touching $v$ from below at $x_0$'' means simply that
     \begin{equation*}
     	  u\leq v \text{ in } \mathbb{R}^d, \text{ and } u(x_0)=v(x_0).
     \end{equation*} 
\end{DEF}

  \begin{thm}[Form of Linear Operator {\cite[Theorem 1.5]{Courrege-1965formePrincipeMaximum}}]\label{thm:LinearCourrege}
	If $I$ is linear and satisfies the global comparison principle of Definition \ref{def: operators with maximum principle}, then $I$ is a linear L\'evy operator of the form
	 \begin{align}
	 	  I(u,x) = &\Tr(A(x)D^2u(x))+(B(x),\nabla u(x))+C(x)u(x),\nonumber\\
         &+\int_{\mathbb{R}^d}(u(x+h)-u(x)-(\nabla u(x),h)\ind_{B_1(0)}(h))\mu(x,dh),\label{eqPW:LevyOp}
	 \end{align}
	where $A$, $B$, $C$ are bounded functions, $A\geq 0$, $C\leq0$, and $\mu$ satisfies
	\begin{equation*}
		\sup_x \int_{\real^d} \min(\abs{h}^2,1)\mu(x, dh) <+\infty.
	\end{equation*}
  \end{thm}

  \begin{rem} Note Theorem \ref{thm:LinearCourrege} does not say that $A,B$ and $C$ are continuous in $x$. For an example of a linear continuous map $I:C^2(\mathbb{R}^d)\to C(\mathbb{R}^d)$ that has the global comparison principle and whose respective coefficients are not continuous, see \cite[Section 1.6]{Courrege-1965formePrincipeMaximum}.

  \end{rem}

  Below, we prove a nonlinear analogue of Theorem \ref{thm:LinearCourrege} dealing with maps $I$ which are not necessarily linear but are at least Fr\'echet-differentiable (see also Remark \ref{rem: Nonlinear Courrege Differentiability Assumption is Technical}).
  \begin{thm}[Representation of Nonlinear Operators]\label{thm: nonlinear Courrege}
    If $I$ is a Fr\'echet differentiable map $C^2(\real^d)\to C(\real^d)$ which satisfies the global comparison principle, then
    \begin{equation}\label{eqPerfectWorld:InfSup}
      I(u,x) = \min\limits_{a} \max\limits_{b}\left(f^{ab}(x)+L^{ab}(u,x)\right)	
    \end{equation}
    where for all $a,b$, $f^{ab}\in C(\real^d)$ and $L^{ab}$ is a L\'evy operator of the form (\ref{eqPW:LevyOp}).
  \end{thm}

  The proof will consist in first obtaining a $\min-\max$ formula for $I$ in terms of linear operators (given by the derivative of $I$), and then showing that those linear operators inherit the global comparison principle from $I$. This is done in Lemma \ref{lem:DifferentiableOnBanach} and Proposition \ref{prop:subdifferential} below.

\begin{rem}\label{rem: Nonlinear Courrege Differentiability Assumption is Technical} The assumption that $I$ be Fr\'echet is technical. Ideally, the $\min-\max$ formula ought to hold when the map $I$ is merely Lipschitz. One way to extend Lemma \ref{lem:DifferentiableOnBanach} below to less regular mappings would be working with the generalized Jacobian of $I$ (in the Clarke sense, see \cite{Cla1990optimization}) instead of the Fr\'echet derivative.

\end{rem}

\begin{lem}\label{lem:DifferentiableOnBanach}
  Given any Fr\'echet differentiable map $I:C^2(\mathbb{R}^d)\to C(\mathbb{R}^d)$ there is a family of linear (bounded) operators $L^{ab}(\cdot,x)$ and functions $g^{ab}(x)\in C(\mathbb{R}^d)$ such that
  \begin{align*}
    I(u,x) = \min \limits_{a} \max \limits_{b} \left \{f^{ab}(x)+L^{ab}(u,x) \right \}.	  
  \end{align*}	  
  Moreover, each  $L^{ab}$ belongs to $\mathcal{D}I$, which is the set defined by  
  \begin{align*}
	  \mathcal{D}I := \{ L \mid L \textnormal{ is the Fr\'echet derivative of } I \textnormal{ at some } v\in C^2(\mathbb{R}^d)\}.
  \end{align*}
\end{lem}

\begin{proof}[Proof of \autoref{lem:DifferentiableOnBanach}]
  For $v\in C^2(\mathbb{R}^d)$ define a nonlinear map $K_v:C^2(\mathbb{R}^d)\to C(\mathbb{R}^d)$ by
  \begin{align*}
    K_v(u,x):=\max \limits_{L \in \mathcal{D}I} \left \{I(v,x)+L(u-v,x) \right \},\;\;u\in C^2(\mathbb{R}^d).
  \end{align*}	  
  This $K_v$ is well defined provided $\mathcal{D}I\neq \emptyset$. To prove the lemma, it suffices to show that
  \begin{align*}
    I(u,x) = \min\limits_{v \in C^2(\mathbb{R}^d)} K_v(u,x),\;\;\forall\;u,\;x.
  \end{align*}
  First off, we clearly have
  \begin{align*}
    I(u,x) = K_u(u,x),\;\;\; \forall\;\;u,\;\;x.
  \end{align*}	
  so,
  \begin{align*}
    I(u,x) \geq \min\limits_{v \in C^2(\mathbb{R}^d)} K_v(u,x),\;\;\; \forall\;\;u,\;\;x.
  \end{align*}	    
  To prove the opposite inequality, fix $x_0\in\mathbb{R}^d$. Then the real-valued map
  \begin{align*}
    u \to I(u,x_0) 	  
  \end{align*}	  
  is Fr\'echet-differentiable. Moreover, if $L_u:C^2(\mathbb{R}^d)\to C(\mathbb{R}^d)$ is the linear operator given by the derivative of $I$ at $u$, then the derivative of the above functional is simply given by the evaluation of the operator $L_u$ at $x_0$, so
  \begin{align*}
    \phi \to L(\phi,x_0).  
  \end{align*}	   
  With this in mind, the mean value theorem (applied to the scalar functional) says that given any $u,v \in C^2(\mathbb{R}^d)$ there is some $L \in \mathcal{D}I$ (depending only $u,v$ and $x_0$) such that
  \begin{align*}
    I(u,x_0)=I(v,x)+L(u-v,x_0).	  
  \end{align*}	  
  Taking the maximum for $L \in \mathcal{D}I$, it follows that
  \begin{align*}
    I(u,x_0)\leq \max \limits_{L \in \mathcal{D}I} \left \{ I(v,x)+L(u-v,x_0)\right \} = K_v(u,x_0),  \forall\;u\in C^2(\mathbb{R}^d),
  \end{align*}	  
  since $x_0$ was arbitrary, the lemma is proved.
\end{proof}

\begin{prop}\label{prop:subdifferential}
  Let $I$ be an operator as in Lemma \ref{lem:DifferentiableOnBanach}. If $I$ has the global comparison property then any linear operator $L\in \mathcal{D} I$ also has the global comparison property.
\end{prop}

\begin{proof}
  Fix $u\in C^2(\mathbb{R}^d)$, and $\phi \in C^2(\mathbb{R}^d),\; x_0 \in\mathbb{R}^d$ such that $\phi$ achieves its global maximum at $x_0$. The goal is to show that
  \begin{align*}
    L(\phi,x_0)\leq 0, \textnormal{ where  } L = L_u.
  \end{align*}
  For $s\in\mathbb{R}$ let $u_s:= u+s(\phi(x)-\phi(x_0))$. Given that $\phi$ achieves its global maximum at $x_0$, it follows that
  \begin{align*}
    u_s(x_0)=u(x_0),\;\; u_s(x)\leq u(x),\;\;\forall\; x\in\mathbb{R}^d,
  \end{align*}
  which holds as long as $s>0$. This means that $u_s$ touches $u$ from below at $x_0$, thus, since $I$ has the global comparison property,
  \begin{align*}
    I(u_s,x_0)\leq I(u,x_0),\;\;\forall\;s>0.	  
  \end{align*}	  
  Since there is equality for $s=0$, the derivative with respect to $s$ at $s=0$ cannot be positive, but this simply says
  \begin{align*}
    L(\phi,x_0) \leq 0,	  
  \end{align*}	  
  and the proposition is proved.
\end{proof}

\begin{proof}[Proof of Theorem \ref{thm: nonlinear Courrege}]

  The theorem follows immediately by combining the previous Proposition and Lemma with Courr\'ege's theorem (Theorem \ref{thm:LinearCourrege}). Indeed, Lemma \ref{lem:DifferentiableOnBanach} says that 
  \begin{align*}
    I(u,x) = \min\limits_a \max \limits_b \{ L^{ab}(u,x)+f^{ab}(x) \}
  \end{align*}
  where each $L^{ab}$ belongs to $\mathcal{D}I$. Then, Proposition \ref{prop:subdifferential} states in particular that each $L^{ab}$ appearing in the min-max formula satisfies the global comparison principle, and thus (by Theorem \ref{thm:LinearCourrege}) it must be of the form \eqref{eqPW:LevyOp}.

%	Since $I$ is Lipschitz, the graph of $I$ can be touched from below and above at every point by cone functions of opening at most $\norm{I}_{C^2_b}$.  Indeed, we see that
%	\[
%	I(u,\cdot) - I(\phi,\cdot) \leq C\norm{u-\phi}_{C^2_b}.
%	\]  
%	Hence we can represent
%	\begin{equation*}
%		I(u,\cdot) = \min_{\phi\in C^2_b}\left( I(\phi,\cdot) + C\norm{u-\phi}_{C^2_b}   \right) =  \min_{\phi\in C^2_b}\left( I(\phi,\cdot) + C\norm{u-\phi}_{(C^2_b)^{**}}   \right).
%	\end{equation*}
%	The second equality in the previous line is a consequence of the fact that $\norm{u-\phi}_{C^2_b}\leq\norm{u-\phi}_{(C^2_b)^{**}}$ by the Hahn-Banach theorem.  Using the dual representation of $\norm{u-\phi}_{(C^2_b)^{**}}$, we have
%	\begin{equation}\label{eqPerfectWorldDualNorm}
%		I(u,x) = \min_{\phi\in C^2_b} \max_{L\in (C^2_b)^*}\left( I(\phi,x) + L(u-\phi,x)   \right),
%	\end{equation}
%	which again by the Hahn-Banach theorem is always a maximum, not just a supremum.
%	We note $I(\phi,x)-L(\phi,x)=f^{ab}(x)$ in (\ref{eqPerfectWorld:InfSup}).
%	We conclude the proof of the Theorem by invoking Lemma \ref{lem:MinMaxGlobalComparison}, which states that all of the linear operators contained in the min-max can be assumed to satisfy the global comparison principle.  We note that without loss of generality, we can assume that 
%	\[
%	I(0,x) = 0\ \ \textnormal{for all}\ x,
%	\] 
%	as we can simply work with 
%	\[
%	I(u,x)-I(0,x),
%	\]
%	instead of $I$, without changing the Lipschitz nature of $I$.
\end{proof}

%\begin{lem}\label{lem:MinMaxGlobalComparison}
%	If $I$	 is an operator which satisfies the global comparison principle and $I$ can be represented by
%	\[
%	I(u,x) = \min\limits_{a} \max\limits_{b}\left(L^{ab}(u,x)\right)		
%	\]
%	for some collection of linear operator $\{L^{ab}\}$, then the collection can be restricted to contain only those $L^{ab}$ which also satisfy the global comparison principle.
%\end{lem}

%\begin{proof}[Proof of Lemma \ref{lem:MinMaxGlobalComparison}]
%	Assume that $u$ and $v$ are $C^2_b$ functions, $u\leq v$, and $u(x_0)=v(x_0)$.  Assume that $a^*$ is the choice of the parameter that gives equality for the min in the definition of $I(u,x_0)$, and then subsequently let $b^*$ be any choice of parameter that gives the maximum for $v$, i.e.
%	\[
%	L^{a^*b^*}(v,x_0) = \max_{b}\left( L^{a^*b}(v,x_0)  \right).
%	\]
%Then invoking the global comparison principle for $I$ and unraveling the min-max, we see that
%	\begin{align*}
%		0 &\geq I(u,x_0) - I(v,x_0)\\
%		&= \min_a\max_b \left( L^{ab}(u,x_0)\right) - \min_{a}\max_{b}\left( L^{ab}(v,x_0) \right)\\
%		&\geq \max_b\left( L^{a^*b}(u,x_0) \right) - \max_b \left( L^{a^*b}(v,x_0) \right)\\
%		&\geq L^{a^*b^*}(u,x_0) - L^{a^*b^*}(v,x_0).
%	\end{align*} 
%\end{proof}

%%%% OLD PRE FEB '15 VERSION *PROOF* ENDS HERE **** 

\begin{rem}
	We note that Theorem \ref{thm: nonlinear Courrege} is not so surprising once Theorem \ref{thm:LinearCourrege} is established, and such min-max representations have been widely used in the analysis of both first and second order (\textbf{local}) nonlinear PDE for decades (e.g. \cite{ConwayHopf-1964HamiltonTheoryGeneralizedSol}, \cite{Evans-1980OnSolvingPDEAccretive}, \cite{Evans-1984MinMaxRepresentations}, \cite{Fleming-1964CauchyProblemDegenerat-UsesMinMax}, \cite{Fleming-1969CauchyProbNonlinFirstOrder-UsesMinMax}, \cite{Hopf-1950ThePDE-ViscConsLawCPAM}, \cite{Katsou-1995RepresentationDegParEq}, \cite{Souganidis-1985MaxMinRep}).  What will be interesting and most likely difficult is to determine the specific structure of the components of $L^{ab}$ in the $\inf\sup$, specifically that of the L\'evy measure, $\mu(x, dh)$.  In particular whether or not $\mu(x,dh)$ has a density and if is the density comparable to a canonical one such as the one corresponding to the Fractional Laplacian, etc... 
\end{rem}

\begin{rem}
	We note that Theorem \ref{thm: nonlinear Courrege} is very much related to the results on monotonicity preserving semigroups and monotonicity preserving interface motion which appeared in \cite{AlvarezLionsGuichardMorel-1993AxiomsImageProARMA}, \cite{BarlesSouganidis-1998NewApproachFrontsARMA}, and  \cite{Biton-2001NonlinMonoSemiViscSolAIHP} which represent these phenomena as the unique viscosity solutions of degenerate parabolic equations.
\end{rem}

Finally, we remark that the Dirichlet-to-Neumann map always has the global comparison property, so that -in particular- it has a min-max representation whenever it is a Fr\'echet differentiable map from $C^2(\Sigma_0)$ to $C(\Sigma_0)$.
\begin{prop}
	If $\I^1:C^2(\Sigma_0) \to C(\Sigma_0)$ is Fr\'echet differentiable, then it admits an $\min-\max$ representation as in (\ref{eqPerfectWorld:InfSup}).
\end{prop}

\begin{proof}
	First we remark that $\I^1$ does indeed obey the global comparison principle.  Supposing that $u$ is touched from above by $v$ at $x_0\in\Sigma_0$, then by the comparison principle for (\ref{eqSetUp:DirichletrScale}), we see that $U^1_u\leq U^1_v$ in all of $\bar \Sigma^1$.  Hence since $U^1_u(x_0)=U^1_v(x_0)$, we conclude that
	\[
	\I^1(u,x_0) = \partial_\nu U^1_u(x_0) \leq U^1_v(x_0) = \I^1(v,x_0).
	\]
	Since $\mathcal{I}^1$ is Fr\'echet differentiable, Theorem \ref{thm: nonlinear Courrege} yields the $\min-\max$ representation.
\end{proof}

Since we don't actually invoke Theorem \ref{thm: nonlinear Courrege} in our proof of Theorem \ref{thm:HomogMain} but only use it for heuristics, we will also mention one last lemma without proof.  There are many more issues related to Lemma \ref{lem:DtoNRepresentationDetails}, which we will develop in a separate work.  The statement of Lemma \ref{lem:DtoNRepresentationDetails} is, however, useful for the heuristics of our line of attack on Theorem \ref{thm: nonlinear Courrege} and so we include it.
\begin{lem}\label{lem:DtoNRepresentationDetails}
	If $\I^1$ is represented as (\ref{eqPerfectWorld:InfSup}) with $L^{ab}$ in (\ref{eqPW:LevyOp}), then $A^{ab}=0$, $C^{ab}=-1$, and $B^{ab}$, $\mu^{ab}$ are both independent of $x$.  In particular
	\begin{align}\label{eqPerfectWorldI1Rep}
		\I^1(u,x) = 
		-u(x) + &\inf_a\sup_b\{ B^{ab}\cdot \nabla u(x) \nonumber\\
		&\ \ + \int_{\Sigma_0}(u(x+h)-u(x)-(\nabla u(x),h)\ind_{B_1(0)}(h))\mu^{ab}(dh) \}.
	\end{align}
\end{lem}

\begin{rem}
	It is important to stress that it is not clear when $\mu^{ab}$ will be symmetric ($\mu^{ab}(-dh)=\mu^{ab}(dh)$) and $B^{ab}=0$.  In particular when $F$ is rotationally invariant, then such properties of $B^{ab}$ and $\mu^{ab}$ can be shown, but not in general.  This is important because there is often a distinction between regularity results for integro-differential operators with symmetric versus non-symmetric L\'evy measures, cf., e.g. \cite{CaSi-09RegularityIntegroDiff}, \cite{KassmannMimica-2013IntrinsicScalingArXiv}, \cite{KassRangSchwa-2013RegularityDirectionalINDIANA} vs. \cite{Chan-2012NonlocalDriftArxiv}, \cite{ChDa-2012NonsymKernels},   \cite{Silv-2006Holder}.
\end{rem}

\begin{rem}
	The reader can see, e.g. \cite[Section 4]{Hsu-1986ExcursionsReflectingBM}  for a similar representation in the linear case of Lemma \ref{lem:DtoNRepresentationDetails} as it pertains the stochastic processes which are a reflected brownian motion, and $\I$ generates the induced boundary process (basically the probabilistic presentation of the Dirichlet-to-Neumann operator).  
\end{rem}

\subsection{The Analysis of (\ref{eqIntro:NonlocalEpsilonEq}) Assuming The Inf-Sup Formula}\label{sec:AssumeInfSup}

We now develop the heuristics which lead to the analysis appearing in section \ref{sec:ThmProofs}.  There are some special features resulting from the translation invariant set-up of (\ref{eqIntro:MainStripDomain}) which allow for simplifications and a proof that does not rely on Theorem \ref{thm: nonlinear Courrege}.  After some of the issues regarding the structure of the L\'evy measures have been resolved (alluded to in Section \ref{sec:NonlinearCourrege}), we believe this method will useful in other contexts such as equations which have oscillations in $F(D^2u^\ep,x/\ep)$ as well as more general domains.

As mentioned in Section \ref{sec:Intro}, the heart of our proof is the analysis of an auxiliary homogenization problem given as
\begin{equation}\label{eqPerfectWorld:e1}
	\I^1(u^\ep|_{\Sigma_0},x) = g(\frac{x}{\ep})\ \text{in}\ \Sigma_0.
\end{equation}
Because $\I^1$ is a nonlocal operator, the strategy for resolving (\ref{eqPerfectWorld:e1}) requires treatment of the global values of test functions, not just local quantities such as the gradient or the Hessian.  Following \cite[Section 2.1]{Schw-10Per}, we explain the relevant corrector (or approximate corrector) equation as it pertains to (\ref{eqPerfectWorld:e1}).

Given the positive 1-homogeneity of $F$, it is straightforward to see that the nonlocal operator of (\ref{eqPerfectWorld:e1}) should have a scaling exponent of $1$.  This is almost true, but the scaling is corrupted slightly due to the fact that $\I^1$ is influenced by the zero Dirichlet condition on $\Sigma_1$.  However, the equation has a scaling of $1$ in the sense that the correct expansion for (\ref{eqPerfectWorld:e1}) uses rescaling of the form 
\[
v\mapsto \ep v(\frac{\cdot}{\ep}).
\]
Therefore, in order to identify an effective operator for (\ref{eqPerfectWorld:e1}) it will be necessary to study for all smooth test functions, $\phi$, a way to balance the oscillations in
\[
\I^1(\phi + \ep v(\frac{\cdot}{\ep}),x) = g(\frac{x}{\ep}).
\]
Now we can use the inf-sup representation of $\I^1$ from Lemma \ref{lem:DtoNRepresentationDetails}-- (\ref{eqPerfectWorldI1Rep})-- to simplify how $\I^1$ acts on functions of the form $\phi + \ep v(\frac{\cdot}{\ep})$.
\begin{align*}
	&\I^1(\phi + \ep v(\frac{\cdot}{\ep}),x) = \\
	&\ \ -\phi(x) - \ep v(\frac{x}{\ep})\\
	&\ \  + \inf_a\sup_b\{ (B^{ab}, \nabla \phi(x)) + \int_{\Sigma_0}(\phi(x+h)-\phi(x)-(\nabla \phi(x),h)\ind_{B_1(0)}(h))\mu^{ab}(dh)\\
	&\ \ + (B^{ab}, \nabla v(\frac{x}{\ep})) + \int_{\Sigma_0}(\ep v(\frac{x+h}{\ep})-\ep v(\frac{x}{\ep})-(\nabla v(\frac{x}{\ep}),h)\ind_{B_1(0)}(h))\mu^{ab}(dh) \}.
\end{align*}
In order to streamline presentation, we introduce a couple of operators:
\[
I^{ab}(u,x) = -u(x) + (B^{ab}, \grad u(x) ) + \int_{\Sigma_0}(u (x+h)-u(x)-(\nabla u(x),h)\ind_{B_1(0)}(h))\mu^{ab}(dh)
\]
and
\[
I^{ab}_\ep(u,x) = -u(x) + (B^{ab}, \grad u(x) ) + \int_{\Sigma_0}(u (x+h)-u(x)-(\nabla u(x),h)\ind_{B_{1/\ep}(0)}(h))\mu^{ab}(\ep^{-1}dh).
\]
We must investigate (after changing variables in the second integration, which applies to $v$!)
\[
\inf_a\sup_b\left\{ I^{ab}(\phi,x) + I^{ab}_\ep(v,\frac{x}{\ep})   \right\} = g(\frac{x}{\ep}). 
\]
A significant simplification arises in this asymptotic problem because $I^{ab}$ is continuous with respect to the $C^{1,\gam}$ norm.  We can therefore effectively separate variables and freeze $x=x_0$, while $y=x/\ep$ is the variable of interest.  This culminates in the search for a unique constant $\lam$ such that there is a classical solution, $v^\ep$ of the equation (which still depends on $\phi$, but in a fixed way)
\begin{equation}\label{eqPerfectWorld:Corrector}
\inf_a\sup_b\left\{ I^{ab}(\phi,x_0) + I^{ab}_\ep(v^\ep,y)   \right\} = g(y) + \lam.
\end{equation}
We emphasize that
$I^{ab}(\phi,x_0)$
are constant with respect to $\ep$ and $y$.  This $v^\ep$ should be the ``corrector'' which balances the equation in $\ep$ near $x_0$ for the global behavior of $\phi$.

More generally one expects that $I^{ab}(\phi,x_0)$ will contribute a term which is a uniformly bounded and uniformly continuous function of $y$ (see \cite[Section 2.1]{Schw-10Per}).

The methods typically used for analyzing the corrector equation, (\ref{eqPerfectWorld:Corrector}), require both existence/uniqueness theory and $C^{\tilde\gam}_{loc}(\Sigma_0)$ regularity results which depend only on universal parameters such as ``ellipticity'' and $L^\infty$ bounds (in the context of integro-differential equations, ``ellipticity'' is not so obviously defined as in the second order theory).  This is the place where our heuristics must stop because without further information on $\mu^{ab}$, we do not know if such results exist. The validity of regularity results for operators-- such as $\I^1$-- realized via Theorem \ref{thm: nonlinear Courrege} is an important and difficult open question.  The existence and uniqueness is not a problem because in the special case treated here, $I^{ab}_\ep$ are translation invariant (see e.g. \cite[Sections 3, 4, 5]{CaSi-09RegularityIntegroDiff}), but in more interesting contexts they are not expected to be.  However, the $C^{\tilde\gam}$ estimates are delicate, and depend on fine properties of the extremal operators, $M^{1,\pm}$, in Lemma \ref{lemSetup:DtoNExtremals}.  This requires one to know some specific upper and lower bounds on $\mu^{ab}$ which are uniform in $a,b$ as well as fit into the existing theory, which can have significantly different assumptions on the bounds for these measures (cf. \cite{CaSi-09RegularityIntegroDiff} vs.  \cite{Chan-2012NonlocalDriftArxiv}, \cite{GuSc-12ABParma}, or \cite{KassRangSchwa-2013RegularityDirectionalINDIANA}).

%%%%%%%%%%%%%%%%%%%%%%%%%%%%%%%%%%%%%%%%%%%%%%
%%%%%%%%%%%%%%%%%%%%%%%%%%%%%%%%%%%%%%%%%%%%%%
%%%%%%%%%%%%%%%%%%%%%%%%%%%%%%%%%%%%%%%%%%%%%%
%%%%%%%%%%%%%%%%%%%%%%%%%%%%%%%%%%%%%%%%%%%%%%
%%%%%%%%%%%%%%%%%%%%%%%%%%%%%%%%%%%%%%%%%%%%%%
\section{ The Proofs of Theorems \ref{thm:HomogMain} and \ref{thm:HomogNonlocalBoundaryEq}}\label{sec:ThmProofs}

%%%%%%%%%%%%%%%%%%%%%%%%%%%%%%%%%%%%%%%%%%%%%%
%%%%%%%%%%%%%%%%%%%%%%%%%%%%%%%%%%%%%%%%%%%%%%
\subsection{Homogenization of $u^\ep$ in $\Bar\Sigma$ (proof of Theorem \ref{thm:HomogMain})}\label{sec:ProofThmMain}

We first provide a proof of Theorem \ref{thm:HomogMain}, assuming Theorem \ref{thm:HomogNonlocalBoundaryEq} is true.  

\begin{proof}[Proof of Theorem \ref{thm:HomogMain}]
	We begin by collecting some facts about $u^\ep$.  Firstly, since $g\in C^\gam(\Sigma_0)$, it follows from \cite[Theorem 8.2]{MiSi-2006NeumannRegularity} (see also Theorem \ref{thm:NonlocalMilakisSilvestreGlobalHolderAndC1gam}) that $u^\ep\in C^{1,\tilde\gam}$ for some $\tilde\gam>0$ possibly smaller than $\gam$, but depending only on universal parameters-- note, $[\grad u^\ep]_{C^{\tilde\gam}}$ does depend on $\ep$ through $g(\cdot/\ep)$. This says that $u^\ep$ attains its normal derivative continuously and classically in (\ref{eqIntro:MainStripDomain}), see Lemma \ref{lem:NeumannVSisClassical}. Furthermore, it also holds that $\norm{u^\ep}_{C^{\tilde\gam}(\bar\Om)}\leq C(\Om)$ independently of $\ep$, for any $\Om\subset\subset \bar \Sigma^{1/2}$ (by \cite[Theorem 8.1]{MiSi-2006NeumannRegularity}, see also Theorem \ref{thm: MilakisSilvestre}).

	Thus if $U^\ep$ is the unique solution of (\ref{eqIntro:Dirichlet}) such that $U^\ep= u^\ep|_{\Sigma_0}$ on $\Sigma_0$, then uniqueness of viscosity solutions tells us that in fact $U^\ep$ solves (\ref{eqIntro:MainStripDomain})-- since the normal derivatives are attained classically-- and hence $U^\ep=u^\ep$ in all of $\Bar\Sigma^1$.  By Theorem \ref{thm:HomogNonlocalBoundaryEq}, $u^\ep|_{\Sigma_0}\to c$ uniformly where $c=\bar\I(0)$. By the stability of (\ref{eqIntro:Dirichlet}) with respect to uniform convergence of boundary data, it follows that $U^\ep\to \Bar U$ uniformly in $\Bar\Sigma$, where $\Bar U$ solves (\ref{eqIntro:Dirichlet}) with Dirichlet date given by the constant $c$ on $\Sigma_0$.
	
	Since $U^\ep=u^\ep$ and (\ref{eqIntro:Dirichlet}) with constant boundary (Dirichlet) data has an explicit unique solution, we conclude that $u^\ep\to l_c$ uniformly in $\bar\Sigma^1$, where $l_c$ is the affine function
	\begin{equation*}
		l_c(X) = c(1-X\cdot\nu).
	\end{equation*}
	Hence the effective Neumann condition is (since $c=\bar\I(0)$)
	\begin{equation*}
		\bar g(\nu) = \partial_\nu l_c = -\bar\I(0),
	\end{equation*}
	and this completes the proof.
\end{proof}

%%%%%%%%%%%%%%%%%%%%%%%%%%%%%%%%%%%%%%%%%%%%%%
%%%%%%%%%%%%%%%%%%%%%%%%%%%%%%%%%%%%%%%%%%%%%%
\subsection{Limit of $u^\ep$ on The Boundary $\Sigma_0$ (Proof of Theorem \ref{thm:HomogNonlocalBoundaryEq})}\label{sec:ProofThmBoundary}

Now we present the proof of Theorem \ref{thm:HomogNonlocalBoundaryEq}.

For the sake of notation when working in the boundary, $\Sigma_0$, we will call $v^\ep: \Sigma_0\to\real$
\begin{equation}\label{eqBoundary:UEpVEpDef}
	v^\ep:= u^\ep|_{\Sigma_0}.
\end{equation}
The corresponding global problem for $v^\ep$ then reads
\begin{equation}\label{eqBoundary:VEpEq}
	\I^1(v^\ep,x) = g(\frac{x}{\ep})\ \ \text{in}\ \ \Sigma_0.
\end{equation}
One last notation we will use is the function which is $v^\ep$ at the microscale:
\begin{equation}\label{eqBoundary:WEpDef}
	w^\ep(y):= \frac{1}{\ep} v^\ep (\ep y),
\end{equation}
which gives the unscaled equation for $w^\ep$
\begin{equation}\label{eqBoundary:WEpMicroscaleEq}
	\I^{1/\ep}(w^\ep,y) = g(y)\ \ \text{in}\ \ \Sigma_0.
\end{equation}

Equation (\ref{eqBoundary:VEpEq}) is an auxiliary homogenization problem which is posed on $\Sigma_0$ only.  It is precisely the feature of this set-up which is the core of our proof of Theorem \ref{thm:HomogNonlocalBoundaryEq}.  We will prove the existence and uniqueness of the constant, $\Bar \I(0)$ separately.  The existence is a consequence of the almost periodicity of $g$, and should be thought of as a nonlocal elliptic modification of the results of \cite{Ishii-2000AlmostPeriodicHJHomog}, which we present in Proposition \ref{prop:AlmostPeriodWEp}, Lemmas \ref{lem:IshiiEpWEpDecay}, \ref{lem:IshiiExistenceOfConst}.  The uniqueness of the constant is a consequence of the ``uniform ellipticity'' of $\I^r$ and appears in Lemma \ref{lem:UniquenessOfConst}.

The key lemma is basically a nonlocal version of the almost periodic arguments which appeared for Hamilton-Jacobi equations in \cite{Ishii-2000AlmostPeriodicHJHomog}.  There are however, many differences between the Hamilton-Jacobi setting and our nonlocal setting here.  

\begin{DEF}\label{def: epsilon-slmost period}
  (See also Remark \ref{rem:AlmostPeriodic}) Given $\phi \in C^{0}(\Sigma_0)$, $y\in \Sigma_0$ and $\delta>0$ we will say that $y$ is a $\delta$-almost period of $\phi$ if
  \begin{align*}
    \|\phi(\cdot+y)-\phi(\cdot)\|_{L^\infty(\Sigma_0)}<\delta	  
  \end{align*}	  
\end{DEF}

\begin{prop}\label{prop:AlmostPeriodWEp}
  Every $\delta$-almost period for $g$ is a $\delta$-almost period for $\ep w^\ep$.
\end{prop}

\begin{proof}
  The function $\tilde w^\ep(y):= w^\ep(y+\tau)$ solves the equation
  \begin{align*}
    \I^{1/\ep}(\tilde w^\ep,y) = g(y+\tau)\;\;\forall\;y\in\Sigma_0.
  \end{align*}
  Then, Lemma \ref{lem:RHSComparison} says that 
  \begin{align*}
    \|\tilde w^\ep - w^\ep\|_{L^\infty(\Sigma_0)} \leq \ep^{-1}\|g(\cdot+\tau)-g(\cdot)\|_{L^\infty(\Sigma)}.	  
  \end{align*}	  
  Since $\tau$ is a $\delta$-period for $g$, this means that $|w^\ep(y+\tau)-w^\ep(y)|\leq \ep^{-1}\delta$ for all $y\in\Sigma_0$.
  
\end{proof}

\begin{lem}[Nonlocal Elliptic Version of Ishii \cite{Ishii-2000AlmostPeriodicHJHomog}]\label{lem:IshiiEpWEpDecay}
	$w^\ep$ from (\ref{eqBoundary:WEpMicroscaleEq}) satisfies the decay
	\begin{equation}\label{BoundaryEq:EpWEpDecay}
		\norm{\ep w^\ep - \ep w^\ep(0)}_{L^\infty(\Sigma_0)} \to 0\ \ \text{as}\ \ \ep\to0.
	\end{equation}
\end{lem}

\begin{proof}[Proof of Lemma \ref{lem:IshiiEpWEpDecay}]
	Let $\{\varepsilon_k\}_{k}$ be a sequence such that $\varepsilon_k \to 0^+$, and let $\{ y_k\}_k $ be a sequence in $\Sigma_0$ such that for each $k$,
	\begin{align*}
	  |\varepsilon_k w^{\varepsilon_k}(y_k)-\varepsilon_k w^{\varepsilon_k}(0)|\geq \tfrac{1}{2}\| \varepsilon_k w^{\varepsilon_k}-\varepsilon_k w^{\varepsilon_k}(0)\|_{L^\infty(\Sigma_0)}.
	\end{align*}
	Let $\delta>0$ be given. Thanks to Remark \ref{rem:AlmostPeriodic}, there is some $R_\delta>0$ such that
	\begin{align*}  
	  \left ( z+ B_{R_\delta}\right ) \cap E_{\delta} \neq \emptyset\;\;\forall\;z\in \Sigma_0,
	\end{align*}
	where $E_\delta$ denotes the set of $\delta$-almost periods for $g$. According to Proposition \ref{prop:AlmostPeriodWEp}, every $\delta$-period for $g$ is also a $\delta$-period for each $\ep w^\ep$, $\varepsilon>0$.
	
	Taking $\varepsilon = \varepsilon_k$, $z=y_k$ above, it follows that for each $k$ there is some $\tau_k$ which is a $\del$-almost period for $\varepsilon_k w^{\varepsilon_k}$ and such that
	\begin{align*}
	  y_k-\tau_k \in B_{R_\delta}.
	\end{align*}
	In particular,
	\begin{align*}
	  |\varepsilon_k w^{\varepsilon_k}(y_k)-\varepsilon_k w^{\varepsilon_k}(0)| & \leq |\varepsilon_k w^{\varepsilon}(y_k)-\varepsilon_k w^{\varepsilon_k}(y_k-\tau_k)|+|\varepsilon_k w^{\varepsilon_k}(y_k-\tau_k)-\varepsilon_k w^{\varepsilon_k}(0)|.
	\end{align*}	
	Since $\tau_k$ is a $\del$-almost period for $\varepsilon_k w^{\varepsilon_k}$ the first quantity on the left is at most $\delta$, 
	\begin{align*}
	  |\varepsilon_k w^{\varepsilon_k}(y_k)-\varepsilon_k w^{\varepsilon_k}(0)| \leq \delta + \osc \limits_{B_{R_\delta}}\{\varepsilon_k w^{\varepsilon_k}\} ,\;\;\forall\; k>0.
	\end{align*}
	Next, note that
	\begin{align*}
	  \osc \limits_{B_{R_\delta}}\{ \varepsilon_k w^{\varepsilon_k} \} = \osc \limits_{B_{\varepsilon_k R_\delta}} \{ v^{\varepsilon_k}\}.
	\end{align*}
	Theorem \ref{thm:NonlocalMilakisSilvestreGlobalHolderAndC1gam} guarantees that there is some $\bar\gamma\in (0,1)$ such that the functions $v^{\varepsilon}$ are $C^{\bar \gamma}$-continuous in $B_1$, uniformly in $\varepsilon$. Therefore (for each fixed $\delta>0$),
	\begin{align*}
	  \lim \limits{\varepsilon \to 0^+}\osc \limits_{B_{\varepsilon R_\delta}} \{ v^{\varepsilon}\} = 0.	
	\end{align*}	
    Given that $\varepsilon_k \to 0$, thus for every large enough $k$ (this possibly depending on $\delta$) we have
	\begin{align*}
	  \tfrac{1}{2}\| \varepsilon_k w^{\varepsilon_k}-\varepsilon_k w^{\varepsilon_k}(0)\|_{L^\infty(\Sigma_0)} \leq |\varepsilon_k w^{\varepsilon_k}(y_k)-\varepsilon_k w^{\varepsilon_k}(0)| \leq 2\delta.   	
	\end{align*}
	That is (as the sequence $\varepsilon_k\to 0^+$ was arbitrary)
	\begin{align*}
	  \limsup \limits_{\varepsilon \to 0^+}	\| \varepsilon w^{\varepsilon}-\varepsilon w^{\varepsilon}(0)\|_{L^\infty(\Sigma_0)}\leq 4\delta,
    \end{align*}		
	letting $\del \to 0^+$, the lemma follows.
\end{proof}

\begin{lem}\label{lem:IshiiExistenceOfConst}
	Given any $\ep_j\to0$, there exists a subsequence, $\ep_j'$ such that $v^{\ep'_j}\to C$ uniformly on $\Sigma_0$, for some constant $C$.
\end{lem}

\begin{proof}[Proof of Lemma \ref{lem:IshiiExistenceOfConst}]
	By Theorem \ref{thm:NonlocalMilakisSilvestreGlobalHolderAndC1gam} we know that $v^\ep\in C^{\tilde\gam}(\Sigma_0)$ for some $0< \tilde\gam < 1$. Thus since $v^\ep$ are uniformly bounded, we can take some subsequence such that $v^{\ep_j}(0)\to C$.
	Furthermore, Lemma \ref{lem:IshiiEpWEpDecay} shows that
	\begin{equation*}
		\norm{v^{\ep_j'} - v^{\ep_j'}(0)}_{L^\infty(\Sigma_0)} \to 0\ \text{as}\ \ep\to0.
	\end{equation*}  
	Hence $v^{\ep_j'}\to C$ uniformly on $\Sigma_0$.   
\end{proof}

\begin{lem}\label{lem:UniquenessOfConst}
	The constant, $C$, of Lemma \ref{lem:IshiiExistenceOfConst} is independent of the sequence, $\ep_j$, and hence unique.
\end{lem}

\begin{proof}[Proof of Lemma \ref{lem:UniquenessOfConst}]
	We first point out that this is another location in the proofs where the residual effect of the Dirichlet condition on $\Sigma_{1/\ep}$ is present in the operator $\I^{1/\ep}$ and causes unnecessary difficulty.  In Section \ref{sec:PerfectWorld} these difficulties would not be present.
	
	Let $c_1$ and $c_2$ be constants such that there are sequences $v^{\ep_j}\to c_1$ and $v^{\ep_k}\to c_2$ uniformly on $\Sigma_0$.  We will establish that
	\[
	c_2\leq c_1,
	\]  
	and since the sequences were arbitrary, this proves the lemma.
	If we rewrite $v^{\ep_j}$ and $v^{\ep_k}$ in the microscale variables, this says that (recall $w^\ep$ from (\ref{eqBoundary:WEpDef}))
	\begin{equation*}
		\ep_j w^{\ep_j} \to c_1\ \ \text{and}\ \ \ep_k w^{\ep_k}\to c_2\ \ \text{uniformly on}\ \Sigma_0.
	\end{equation*}
	We will also define the functions 
	\begin{equation*}
		 \hat w^{\ep_j}= w^{\ep_j} - \frac{1}{\ep_j}c_1\ \ \text{and}\ \ \hat w^{\ep_k} = w^{\ep_k} - \frac{1}{\ep_k}c_2.
	\end{equation*}
	In anticipation of applying Lemma \ref{lem:IrComparison}, we need to make sure that $\hat w^{\ep_j}$ and $\hat w^{\ep_k}$ are non-negative.  We do so by shifting them up by respectively $\del_j$, $\del_k$ where
	\begin{equation*}
		\del_j =\norm{\hat w^{\ep_j}}\ \ \text{and}\ \ \del_k=\norm{\hat w^{\ep_k}},
	\end{equation*}
	which gives
	\begin{equation*}
		\hat w^{\ep_j}+\del_j\geq0\ \text{and}\ \hat w^{\ep_k}+\del_k\geq0.
	\end{equation*}
	
	We will assume without loss of generality that $j$ and $k$ are such that $\ep_j<\ep_k$, which is not a problem since $j$ and $k$ can otherwise be chosen independently of one another.  Using the equations for $w^{\ep_j}$ and $w^{\ep_k}$, we see that from Lemmas \ref{lem:AlVAladdConstants} and \ref{lem:IrComparison}
	\begin{align*}
		-\ep_k\del_k + c_2 + g(y) &= \I^{1/\ep_k}(\hat w^{\ep_k}+\del_k,y)\\
		&\leq \I^{1/\ep_j}(\hat w^{\ep_k} + \del_k,y). 
	\end{align*}
	But on the other hand,
	\begin{align}
		-\ep_j\del_j+ c_1 + g(y) &= \I^{1/\ep_j}(\hat w^{\ep_j}+\del_j,y).
	\end{align}
	Thus Lemma \ref{lem:RHSComparison} tells us that
	\begin{align*}
		(-\ep_k\del_k + c_2) - (-\ep_j\del_j + c_1) &\leq \ep_j \sup_{\Sigma_0}((\hat w^{\ep_k}+\del_k) - (\hat w^{\ep_j}+\del_j)).
	\end{align*}
	Hence
	\begin{align}
		c_2 - c_1 &\leq \ep_j\sup_{\Sigma_0} \hat w^{\ep_k} - \ep_j\inf_{\Sigma_0} \hat w^{\ep_j} +\ep_j\del_k - \ep_j\del_j + \ep_k\del_k-\ep_j\del_j \nonumber\\
		 &\leq \ep_j \norm{\hat w^{\ep_k}}_{L^\infty(\Sigma_0)} + \ep_j\norm{\hat w^{\ep_j}}_{L^\infty(\Sigma_0)} + \ep_j\del_k + \ep_k\del_k  \nonumber\\
		 &\leq 2\ep_k \norm{\hat w^{\ep_k}}_{L^\infty(\Sigma_0)} + 2\ep_j\norm{\hat w^{\ep_j}}_{L^\infty(\Sigma_0)}, \label{eqBoundary:Lem4.5e1}
	\end{align}
	where we have used both $\ep_j\del_j\geq0$ and $\ep_j<\ep_k$.  Now, we preserve $\ep_j<\ep_k$ and allow $j\to\infty$ followed by $k\to\infty$.  By construction of $\hat w^{\ep_j}$ and $\hat w^{\ep_k}$, we have
	\[
	\ep_k \norm{\hat w^{\ep_k}}_{L^\infty(\Sigma_0)}\to 0\ \ \text{and}\ \ \ep_j \norm{\hat w^{\ep_j}}_{L^\infty(\Sigma_0)}\to0.
	\]
	Hence $c_2\leq c_1$.  Reversing the roles of $c_1$ and $c_2$ finishes the lemma.

\end{proof}

%%%%%%%%%%%%%%%%%%%%%%%%%%%%%%%%%%%%%%%%%%%%%%
%%%%%%%%%%%%%%%%%%%%%%%%%%%%%%%%%%%%%%%%%%%%%%
%%%%%%%%%%%%%%%%%%%%%%%%%%%%%%%%%%%%%%%%%%%%%%
%%%%%%%%%%%%%%%%%%%%%%%%%%%%%%%%%%%%%%%%%%%%%%
%%%%%%%%%%%%%%%%%%%%%%%%%%%%%%%%%%%%%%%%%%%%%%

\appendix

%%%%%%%%%%%%%%%%%%%%%%%%%%%%%%%%%%%%%%%%%%%%%%
%%%%%%%%%%%%%%%%%%%%%%%%%%%%%%%%%%%%%%%%%%%%%%
%%%%%%%%%%%%%%%%%%%%%%%%%%%%%%%%%%%%%%%%%%%%%%
%%%%%%%%%%%%%%%%%%%%%%%%%%%%%%%%%%%%%%%%%%%%%%
%%%%%%%%%%%%%%%%%%%%%%%%%%%%%%%%%%%%%%%%%%%%%%

\section{Existence and Uniqueness of Viscosity Solutions}\label{sec:ExAndUn}

In this section we collect a few well known facts regarding existence and uniqueness for viscosity solutions of uniformly elliptic problems with either Dirichlet or Neumann boundary conditions.  We point out the sign convention we use is that of \cite{CaCa-95}, which differs by a sign from that of e.g. \cite{CrIsLi-92}.

We let $G$ be a generic fully nonlinear operator (as $F$ was already used).  It could be $F$ or $\M^\pm$, or other examples used in this note. We are concerned with two types of boundary value problems. The first is the Dirichlet problem,

\begin{equation}\label{eqAppend:Dirichlet}
	\left \{ \begin{array}{rll}
		G(D^2W) & = 0 \ &\text{in}\ \Sigma^{r},\\
		W & = 0\ &\text{on}\ \Sigma_{r},\\	
		W & = \phi\ &\text{on}\ \Sigma_0.		
	\end{array}\right.
\end{equation}
The second is the Neumann Problem,
\begin{equation}\label{eqAppend:Neumann}
	\left \{ \begin{array}{rll}
		G(D^2W) & = 0 \ &\text{in}\ \Sigma^{r},\\
		W & = 0\ &\text{on}\ \Sigma_{r},\\	
		\partial_\nu W & = h\ &\text{on}\ \Sigma_0.		
	\end{array}\right.
\end{equation}

\begin{thm}\cite[Theorems VI.3, VI.5]{IshiiLions-1990ViscositySolutions2ndOrder}\label{thm:IshiiLionsExUnDirichletAndNeumann}
	If $G$ is uniformly elliptic as in (\ref{eqSetUp:EllipticF}) and $h$ and $\phi$ are both continuous on $\Sigma_0$, then both (\ref{eqAppend:Dirichlet}) and (\ref{eqAppend:Neumann}) have existence of unique viscosity solutions.
\end{thm}

In general viscosity solutions do not attain their boundary values in a classical way, and so one does not have simply the boundary inequalities in the definitions of viscosity sub / super solutions.  Rather, either the boundary condition holds OR the equation holds (see \cite[Section 7]{CrIsLi-92}, \cite[Section VI]{IshiiLions-1990ViscositySolutions2ndOrder}).  However in the \emph{uniformly elliptic} case, this strange behavior is not present, and the viscosity solution inequalities are exactly what one would expect (implicit, but not explained in \cite[Section 2]{MiSi-2006NeumannRegularity}).  We make this statement precise in the next proposition, which is basically a restatement of \cite[Proposition 7.11]{CrIsLi-92} in the simple context of (\ref{eqAppend:Neumann}).

\begin{prop}\label{propAppend:VSNeumannCondition}
	Assume $G$ is uniformly elliptic in the sense of (\ref{eqSetUp:EllipticF}) and that $h$ is continuous. If $W$ is a viscosity subsolution of (\ref{eqAppend:Neumann}) (in the sense of \cite[Section VI]{IshiiLions-1990ViscositySolutions2ndOrder}), then for all $\phi\in C^2(\bar \Sigma^r)$ which touch $W$ from above at $x_0\in\Sigma_0$, it holds that
	\[
	\partial_\nu \phi(x_0) \geq h(x_0).
	\]
	(That is $W$ satisfies the boundary condition in the strong sense \cite[Definition 7.1]{CrIsLi-92})
\end{prop}

\begin{proof}[Proof of Proposition \ref{propAppend:VSNeumannCondition}]
	Because we are more or less presenting the argument of \cite[Proposition 7.11]{CrIsLi-92}, we only provide the main points.  The key idea is that we can modify the original test function, $\phi$, to a new text function, $\tilde\phi$, which again touches $W$ from above at $x_0$, yet will also be chosen to satisfy $G(D^2\tilde\phi(x_0))<0$.  The definition of viscosity subsolution (modified for our sign convention) requires that
	\begin{equation}\label{eqAppend:ViscInequalityTildePhi}
	\max\left(\partial_\nu \tilde\phi(x_0) - h(x_0),\ G(D^2\tilde\phi(x_0))     \right) \geq 0.
	\end{equation}
	Hence by construction we can force
	\[
	\partial_\nu \tilde\phi(x_0) - h(x_0) \geq 0,
	\]
	and we will also choose $\tilde \phi$ so that $\partial_\nu \tilde\phi$ is arbitrarily close to $\partial_\nu\phi$.
	
	We will use the function 
	\[
	\psi(X) = 
	\begin{cases}
		-\eta (x_{d+1}-1)^2 + (\lam - 2\eta) x_{d+1} + \eta\  &\text{if}\ x_{d+1}\leq \lam/(2\eta)\\
		\lam^2/(4\eta)\ &\text{if}\ x_{d+1} > \lam/(2\eta).
	\end{cases}
	\]
	Here $\lam>0$ and $\eta>0$ are arbitrary.
	Although $\psi$ is not in $C^2(\bar \Sigma^r)$, it is $C^2$ in a neighborhood of $\Sigma_0$, which is good enough.  The construction of $\psi$ shows that for 
	\[
	\tilde\phi(X) = \phi(X)+\psi(X),
	\]
	$\tilde\phi$ also touches $W$ from above at $x_0$.  Furthermore
	\[
	\partial_\nu\tilde\phi(x_0) = \partial_\nu \phi(x_0) + \lam\ \ \text{and}\ \ D^2\tilde\phi(x_0) = D^2\phi(x_0) -\eta \nu\otimes\nu,
	\]
	which guarantees that $\eta$ can be chosen depending only on $D^2\phi$ to give (via \ref{eqSetUp:EllipticF})
	\[
	G(D^2\tilde\phi(x_0)) < 0.
	\] 
	Thus we conclude from (\ref{eqAppend:ViscInequalityTildePhi}) that 
	\[
	\partial_\nu \phi(x_0) - h(x_0) \geq -\lam,
	\]
	and since $\lam>0$ was arbitrary, we conclude the proposition.
	
\end{proof}

\noindent
Proposition \ref{propAppend:VSNeumannCondition} is useful because it tells us that solutions of (\ref{eqAppend:Neumann}) with $C^\gam$ Neumann data will attain their Neumann data classically, which is important for the use of our nonlocal operators $\I^r$.

\begin{lem}\label{lem:NeumannVSisClassical}
	If $h\in C^\gam(\Sigma_0)$ and $W$ is the unique viscosity solution of (\ref{eqAppend:Neumann}) then $\partial_\nu W(x)=h(x)$ classically for all $x\in\Sigma_0$.
\end{lem}

\begin{proof}[Proof of Lemma \ref{lem:NeumannVSisClassical}]
	(Some notation from \cite[Sections 2, 7]{CrIsLi-92} will be used for $J^{2,\pm}$.)  We will use the subsolution property of $W$ to show that classically in $\Sigma_0$,
	\[
		\partial_\nu W \geq h,
	\]
	and the reverse inequality follows by a similar argument which invokes the supersolution property.

		We see from Theorem \ref{thm:MilakisSilvestreGlobalHolderAndC1gam} that $W\in C^{1,\tilde\gam}(\bar\Sigma^{r})$, in particular $W$ is uniformly continuous in $\bar\Sigma^r$.  The uniform continuity implies that the sup-convolution of $W$ (see e.g. \cite[Section II]{IshiiLions-1990ViscositySolutions2ndOrder}) converges \emph{uniformly} to $W$, and hence implies that $W$ in fact has a second order Taylor expansion \emph{from above} on a dense subset of $\bar \Sigma^r$. (Indeed if $W^\al$ is the sup-convolution, then it has a second order Taylor expansion a.e. in $\bar\Sigma^r$ and hence a.e. can be strictly touched from above by a $C^2$ function.  This strict touching from above can be passed to a local touching from above to $W$ at some nearby point.)
	
	 Let $x_0\in\Sigma_0$.  We thus have $X_n\in\Sigma^r$ with $X_n\to x_0$ and $W$ can be touched from above by a $C^2$ function at $X_n$.  Hence there are $(p_n,A_n)\in J^{2,+}_{\Sigma^r}W(X_n)$ and furthermore since $W$ is differentiable,
	\[
	p_n = DW(X_n).
	\]
	Taking limits as $X_n\to x_0$, we see that there is some $(p,A)\in\bar J^{2,+}_{\bar \Sigma^r}W(x_0)$ with 
	\[
	p=DW(x_0),
	\]
	due to the continuity of $DW$.
	Hence by the definition of viscosity subsolution (in e.g. \cite[Section 7]{CrIsLi-92}) and Proposition \ref{propAppend:VSNeumannCondition} we have
	\[
	\partial_\nu W(x_0) = p\cdot\nu \geq h(x_0).
	\]
	Since $x_0$ was arbitrary, we conclude the inequality on all of $\Sigma_0$.  
\end{proof}

If instead of a strip $\Sigma^r$ we consider a half-space (think $r \to \infty$) then a subtlety arises in terms of uniqueness: given a solution to the problem in the half-space one may add a linear function which vanishes along the boundary hyperplane of the half-space (for the Dirichlet problem) or a constant (for the Neumann problem).  We incorporate these observations in the next lemma.

\begin{lem}\label{lem:UniquenessBoundedVS}
	Assume $G$ is uniformly elliptic and positively 1-homogeneous ((\ref{eqSetUp:EllipticF}), (\ref{eqSetUp:1HomF})) and that $w$ and $h$ are bounded and continuous on $\Sigma_0$. There is a unique bounded viscosity solution of
	\begin{equation}\label{eqAppend:DirichletInfinity}
		\begin{cases}
			\displaystyle	G(D^2W) = 0 \ &\text{in}\ \Sigma^\infty=\{x\cdot\nu>0\}\\
			\displaystyle	W=w\ &\text{on}\ \Sigma_0.
		\end{cases}
	\end{equation}
    If there exists a bounded viscosity solution of
	\begin{equation}\label{eqAppend:NeumannInfinity}
		\begin{cases}
			\displaystyle	G(D^2W) = 0 \ &\text{in}\ \Sigma^\infty=\{x\cdot\nu>0\}\\
			\displaystyle	\partial_\nu W= h \ &\text{on}\ \Sigma_0,
		\end{cases}
	\end{equation}
	then it must be unique, up to an additive constant.
\end{lem}

\begin{proof}
	First we look at (\ref{eqAppend:DirichletInfinity}). We just note that existence is not an issue as we can simply extract local uniform limits from the solutions in the domains $\Sigma^r$ as $r\to\infty$ (thanks to the assumption that $w$ is bounded).
    Now we demonstrate the uniqueness. Let $W_1,W_2$ be two bounded solutions of (\ref{eqAppend:DirichletInfinity}) then $\tilde W := W_1-W_2$ satisfies
	\begin{equation*}
		\begin{cases}
			\displaystyle	M^+(D^2\tilde W) \geq 0 \ &\text{in}\ \Sigma^\infty\\
			\displaystyle	M^-(D^2\tilde W) \leq 0 \ &\text{in}\ \Sigma^\infty\\            
			\displaystyle	\tilde W = 0 \ &\text{on}\ \Sigma_0.
		\end{cases}
	\end{equation*}
    Since $\tilde W$ is bounded in $\Sigma^\infty$ and arguing as in the proof of Lemma \ref{lem:LiovillePropForInfiniteDomain}, the oscillation lemma (\cite[Proposition 4.10]{CaCa-95}) can be used to show that $\tilde W$ is a constant. Thus $W\equiv 0$ since it vanishes on $\Sigma_0$. Therefore $W_1=W_2$ in this case. 
	
	If $W_1$, $W_2$ are two bounded solutions to (\ref{eqAppend:NeumannInfinity}), it follows that $\tilde W = W_1-W_2$ solves
	\begin{equation*}
		\begin{cases}
			\displaystyle	M^+(D^2\tilde W) \geq 0 \ &\text{in}\ \Sigma^\infty\\
			\displaystyle	M^-(D^2\tilde W) \leq 0 \ &\text{in}\ \Sigma^\infty\\            
			\displaystyle	\partial_n \tilde W = 0 \ &\text{on}\ \Sigma_0.
		\end{cases}
	\end{equation*}    
    Since $\tilde W$ is a bounded function we can again apply the oscillation lemma (this time \cite[Section 8, equation 8.2]{MiSi-2006NeumannRegularity}) for the Neumann problem and conclude that the oscillation of $\tilde W$ in $\Sigma^\infty$ must vanish.  Thus $W_1-W_2$ is a constant.
    
\end{proof}

%%%%%%%%%%%%%%%%%%%%%%%%%%%%%%%%%%%%%%%%%%%%%%
%%%%%%%%%%%%%%%%%%%%%%%%%%%%%%%%%%%%%%%%%%%%%%
%%%%%%%%%%%%%%%%%%%%%%%%%%%%%%%%%%%%%%%%%%%%%%
%%%%%%%%%%%%%%%%%%%%%%%%%%%%%%%%%%%%%%%%%%%%%%
%%%%%%%%%%%%%%%%%%%%%%%%%%%%%%%%%%%%%%%%%%%%%%

\section{Estimates for the Dirichlet and Neumann problems}

The regularity theory for the Dirichlet and Neumann problems in the fully non-linear setting has a vast literature. The first interior a priori estimates were derived in  \cite{KrSa-1979EstimateDiffusionHitting}, \cite{KrSa-1980PropertyParabolicEqMeasurable} and later extended to viscosity solutions, see \cite{CaCrKoSw-96}, \cite{CaCa-95} for further discussion of known results. The Neumann problem for fully non-linear equations has been widely studied, including Monge-Amp\`ere and Bellman equations \cite{LiTr-1986MongeAmpere},\cite{LiTr-1986Belmann} and other non-linear Neumann-type boundary conditions \cite{Barles-1999NeumannQuasilinJDE}, \cite{BarlesDaLio-2006CAlphaNeumannJDE} . Boundary estimates for the Neumann problem are studied in \cite{MiSi-2006NeumannRegularity} by a reflection technique combined with the techniques used to obtain interior estimates \cite{CaCa-95}, \cite{KrSa-1979EstimateDiffusionHitting}.

For the reader's convenience, we will record here the various regularity results (both for the Dirichlet and Neumann problems) which are needed in our work.  For simplicity we make the references to \cite{MiSi-2006NeumannRegularity} and try to use a similar presentation.  We use the notation for balls whose centers are in $\Sigma_0$ as 
\[
\Sigma^\infty = \left\{ X\in\real^{d+1}\ :\ X\cdot\nu > 0    \right\}
\]
\[
B^+_r(0) = \left\{X\in\Sigma^\infty\ :\ \abs{X}<r  \right\}
\]
and
\[
B'_r(0) = \left\{x\in\Sigma_0\ :\ \abs{X}<1   \right\}.
\]

\begin{thm}\cite[See Theorems 2.1, 2.2]{MiSi-2006NeumannRegularity}\label{thm:MilakisSilvestreDirichletRegularity} 
    Let $U: \overline B_{1/2}^+ \to \mathbb{R}$ be a viscosity solution of 
    \begin{equation*}
	    \left \{ \begin{array}{rll}
              F(D^2U) & = 0   & \text{ in } B_{1/2}^+,\\
               U & = \phi  & \text{ on } B_{1/2}'.		
	    \end{array}\right.
    \end{equation*}
    There are constants $C>0, \tilde\gamma \in (0,\gamma)$, determined by $d,\lambda,\Lambda$ and $\gamma\in (0,1)$ such that
	\begin{itemize}
		\item[(i)]
    \begin{equation*}
        \|U\|_{C^{\tilde\gamma}(B^+_{1/4})} \leq C \left (\|U\|_{L^\infty(B^+_{1/2})}+\|\phi\|_{C^{\gam}(B_{1/2}')}+|F(0)| \right ).
    \end{equation*}
		\item[(ii)]
    \begin{equation*}
        \|U\|_{C^{1,\tilde\gamma}(B^+_{1/4})} \leq C \left (\|U\|_{L^\infty(B^+_{1/2})}+\|\phi\|_{C^{1,\gam}(B_{1/2}')}+|F(0)| \right ).
    \end{equation*}
	\end{itemize}
	
\end{thm}

The estimates gathered in Theorem \ref{thm:MilakisSilvestreDirichletRegularity} lead in a straightforward manner to global estimates for the fully nonlinear Dirichlet problem.

\begin{thm}\label{thm:MilakisSilvestreGlobalHolderAndC1gamDirichlet} 
    Assume that $r\geq 1/2$.  Let $U: \overline \Sigma^r \to \mathbb{R}$ be a viscosity solution of 
    \begin{equation*}
	    \left \{ \begin{array}{rll}
              F(D^2U) & = 0   & \text{ in } \Sigma^r,\\
               U & = \phi  & \text{ on } \Sigma_0.		
	    \end{array}\right.
    \end{equation*}
    There are constants $C>0, \tilde\gamma \in (0,\gamma)$, determined by $d,\lambda,\Lambda$ and $\gamma\in (0,1)$ such that
	\begin{itemize}
		\item[(i)]
    \begin{equation*}
        \|U\|_{C^{\tilde\gamma}(\Bar \Sigma^{r})} \leq C \left (\|U\|_{L^\infty(\bar \Sigma^r)}+\|\phi\|_{C^{\gam}(\Sigma_0)}+|F(0)| \right ).
    \end{equation*}
		\item[(ii)]
    \begin{equation*}
        \|U\|_{C^{1,\tilde\gamma}(\Bar \Sigma^{r})} \leq C \left (\|U\|_{L^\infty(\bar \Sigma^r)}+\|\phi\|_{C^{1,\gam}(\Sigma_0)}+|F(0)| \right ).
    \end{equation*}
	\end{itemize}

\end{thm}

\begin{proof}
    Let $x_0 \in \Sigma_0$ be arbitrary. Since $r\geq 1/2$ we have $B^+_{1/2}(x_0)\subset \Sigma^{r}$, thus $U(\cdot+x_0)$ is a viscosity solution in $B^+_{1/2}$. In that case, Theorem \ref{thm:MilakisSilvestreDirichletRegularity} says that
    \begin{equation*}
		  \|U\|_{C^{\tilde \gamma}(B^+_{1/4}(x_0))}= \| U(\cdot+x_0)\|_{C^{\tilde \gamma}(B^+_{1/4})}\leq C\left ( \|U\|_{L^\infty(\Sigma^r)}+\|\phi\|_{C^{\gamma}(\Sigma_0)}+|F(0)| \right ).
    \end{equation*}
    Now, let $X_0 \in \Sigma^{r/2}$ be such that $B_{1/2}(X_0)\subset \Sigma^r$. Then,
    \begin{equation*}
		  \|U\|_{C^{\tilde \gamma}(B_{1/4}(X_0))} \leq C \left (\|U\|_{L^\infty(B_{1/2}(X_0))}+|F(0)| \right )\leq C\left (\|U\|_{L^\infty(\Sigma^r)}+|F(0)|\right ).
    \end{equation*}
    Since $\{B_{1/2}^+(x_0)\}_{x_0\in \Sigma_0}$ and $\{ B_{1/2}(X_0) \}_{B_{1/2}(X_0)\subset \Sigma^r}$ form a cover of $\Sigma^{3r/4}$, the previous two estimates imply that
    \begin{equation*}
		  \|U\|_{C^{\tilde \gamma}(\Sigma^{3r/4})} \leq C\left ( \|U\|_{L^\infty(\Sigma^r)}+\|\phi\|_{C^{\gamma}(\Sigma_0)}+|F(0)| \right ).
    \end{equation*}
    An estimate for the $C^{\tilde \gamma}$ estimate norm over $\Sigma^r\setminus \Sigma^{r/4}$ can be obtained in the same manner, using the boundary estimate from  Theorem \ref{thm:MilakisSilvestreDirichletRegularity} with $\phi = 0$ since $U \equiv 0 $ on $\Sigma_r$. Combining these two bounds we obtain the first estimate. The second estimate is proven in an analogous manner, using instead the $C^{1,\tilde \gamma}$ boundary estimate from Theorem \ref{thm:MilakisSilvestreDirichletRegularity} and the interior $C^{1,\tilde \gamma}$ estimate for translation invariant fully non-linear equations, e.g. \cite[Corollary 5.7]{CaCa-95}.
    
\end{proof}

As a consequence of the preceding results we see that the operator $\I^r$ maps $C^{1,\gamma}$ to $C^{\tilde \gamma}$ for some $\tilde \gamma \in (0,\gamma)$.

\begin{cor}\label{cor: ClassicalEvaluation}
    For $\gamma \in (0,1)$ there $\exists$ $C>0$, $\tilde \gamma \in (0,\gamma)$ given by $d,\lambda$ and $\Lambda$ such that
    \begin{equation*}
         \|\I^r(\phi)\|_{C^{\tilde \gamma}} \leq C \|\phi\|_{C^{1,\gamma}},\;\;\forall\;\phi \in C^{1,\gamma}(\Sigma_0),\;\forall\;r\geq 1/2.
    \end{equation*}
\end{cor}

\begin{proof}
    Let $U=U^r_\phi$ as defined in \eqref{eqSetUp:DirichletrScale}. Recall that for the $F$ in this problem we are assuming that $F(0)=0$, which means that constants are solutions for the elliptic equation in $\Sigma^r$, thus
    \begin{equation*}
          \|U\|_{L^\infty(\Sigma^r)}\leq \|U\|_{L^\infty(\Sigma_0)}=\|\phi\|_{L^\infty(\Sigma_0)}.
    \end{equation*}
    Then, part two of Theorem \ref{thm:MilakisSilvestreGlobalHolderAndC1gamDirichlet} says that
    \begin{align*}
          \|U\|_{C^{1,\tilde \gamma}(\Sigma^r)} & \leq  C\left (\|U\|_{L^\infty(\Sigma^r)}+ \|\phi\|_{C^{1,\gamma}(\Sigma_0)}+|F(0)|\right ),\\
                                            & \leq  C\left (\|\phi\|_{L^\infty(\Sigma_0)}+ \|\phi\|_{C^{1,\gamma}(\Sigma_0)} \right ),\\
                                            & \leq  2C \|\phi\|_{C^{1,\gamma}(\Sigma_0)}.                                            
    \end{align*}
    Here we used again that $F(0)=0$. Then, 
    \begin{align*}
          \|\I^r(\phi)\|_{C^{\tilde \gamma}(\Sigma_0)} & = \|\partial_\nu U \|_{C^{\bar \gamma}(\Sigma_0)} \leq \| U \|_{C^{1,\bar \gamma}(\Sigma^r)}\leq 2C \|\phi\|_{C^{1,\gamma}(\Sigma_0)},
    \end{align*}
    which is what we wanted.
\end{proof}

Now we review the available results for the Neumann problem, again borrowing the results from \cite{MiSi-2006NeumannRegularity}.

\begin{thm}\cite[See Theorems 8.1, 8.2]{MiSi-2006NeumannRegularity}\label{thm: MilakisSilvestre} 
    Let $U: \overline B_{1/2}^+ \to \mathbb{R}$ be a viscosity solution of 
    \begin{equation*}
	    \left \{ \begin{array}{rll}
              F(D^2U) & = 0   & \text{ in } B_{1/2}^+,\\
              \partial_\nu U & = h  & \text{ on } B_{1/2}'.		
	    \end{array}\right.
    \end{equation*}
    There are constants $C>0, \tilde\gamma \in (0,\gamma)$, determined by $d,\lambda,\Lambda$ and $\gamma\in (0,1)$ such that
	\begin{itemize}
		\item[(i)]
    \begin{equation*}
        \|U\|_{C^{\tilde\gamma}(B^+_{1/4})} \leq C \left (\|U\|_{L^\infty(B^+_{1/2})}+\|h\|_{L^\infty(B_{1/2}')}+|F(0)| \right ).
    \end{equation*}
		\item[(ii)]
    \begin{equation*}
        \|U\|_{C^{1,\tilde\gamma}(B^+_{1/4})} \leq C \left (\|U\|_{L^\infty(B^+_{1/2})}+\|h\|_{C^{\gam}(B_{1/2}')}+|F(0)| \right ).
    \end{equation*}
	\end{itemize}

\end{thm}

The estimates from Theorems \ref{thm:MilakisSilvestreDirichletRegularity} and \ref{thm: MilakisSilvestre} can be used to prove global estimates for $U$. The proof is entirely analogous to that of Theorem \ref{thm:MilakisSilvestreGlobalHolderAndC1gamDirichlet} and we omit it.

\begin{thm}\label{thm:MilakisSilvestreGlobalHolderAndC1gam} 
    Assume that $r\geq 1/2$.  Let $U: \overline \Sigma^r \to \mathbb{R}$ be a viscosity solution of 
    \begin{equation*}
	    \left \{ \begin{array}{rll}
              F(D^2U) & = 0   & \text{ in } \Sigma^r,\\
              \partial_\nu U & = h  & \text{ on } \Sigma_0.		
	    \end{array}\right.
    \end{equation*}
    There are constants $C>0, \tilde\gamma \in (0,\gamma)$, determined by $d,\lambda,\Lambda$ and $\gamma\in (0,1)$ such that
	\begin{itemize}
		\item[(i)]
    \begin{equation*}
        \|U\|_{C^{\tilde\gamma}(\Bar \Sigma^{r})} \leq C \left (\|U\|_{L^\infty(\bar \Sigma^r)}+\|h\|_{L^\infty(\Sigma_0)}+|F(0)| \right ).
    \end{equation*}
		\item[(ii)]
    \begin{equation*}
        \|U\|_{C^{1,\tilde\gamma}(\Bar \Sigma^{r})} \leq C \left (\|U\|_{L^\infty(\bar \Sigma^r)}+\|h\|_{C^\gam(\Sigma_0)}+|F(0)| \right ).
    \end{equation*}
	\end{itemize}

\end{thm}

The above theorem immediately implies estimates for the equation $\I^r(w) = h$ in $\Sigma_0$.

\begin{thm}\label{thm:NonlocalMilakisSilvestreGlobalHolderAndC1gam} 
    Assume that $r\geq 1/2$.  Let $w: \Sigma_0 \to \mathbb{R}$ be a classical solution of 
    \begin{equation*}
		\I^r(w,y) = h(y)\ \ \text{in}\ \Sigma_0.
    \end{equation*}
    There are constants $C>0, \tilde\gamma \in (0,\gamma)$, determined by $d,\lambda,\Lambda$ and $\gamma\in (0,1)$ such that
	\begin{itemize}
		\item[(i)]
    \begin{equation*}
        \|w\|_{C^{\tilde\gamma}(\Sigma_0)} \leq C \left (\|w\|_{L^\infty(\Sigma_0)}+\|h\|_{L^\infty(\Sigma_0)} \right ).
    \end{equation*}
		\item[(ii)]
    \begin{equation*}
        \|w\|_{C^{1,\tilde\gamma}(\Sigma_0)} \leq C \left (\|w\|_{L^\infty( \Sigma_0)}+\|h\|_{C^\gam(\Sigma_0)} \right ).
    \end{equation*}
	\end{itemize}

\end{thm}

Another tool needed in Section \ref{sec:ProofThmBoundary} is the following Liouville theorem.

\begin{lem}\label{lem:LiovillePropForInfiniteDomain}
	If $W$ is a bounded viscosity solution of
	\begin{equation*}%\label{eqHomog:NeumannInfinityDomain}
		\begin{cases}
			\displaystyle	F(D^2W) = 0 \ &\text{in}\ \Sigma^\infty=\{X\cdot\nu>0\},\\
			\displaystyle	\partial_\nu W=0\ &\text{on}\ \Sigma^{\infty}_0,
		\end{cases}
	\end{equation*}
	then $W$ is a constant.
\end{lem}

\begin{proof}
     This is a straightforward consequence of the oscillation lemma for the Neumann problem. We recall that \cite[Section 8, equation 8.2]{MiSi-2006NeumannRegularity}, there exists a $\mu \in (0,1)$ which is determined by $d,\lambda$ and $\Lambda$ such that if $r>0$ and $x_0 \in \Sigma_0$, then 
    \begin{equation*}
         \osc \limits_{B_r^+(x_0)} W  \leq \mu \osc \limits_{B_{2r}^+(x_0)} W.
    \end{equation*}
    Taking $r=1$ and applying the above estimate successively to balls of radii $2^k$ yields
    \begin{equation*}
         \osc \limits_{B_{1}^+(x_0)} W  \leq \mu^k \osc \limits_{B_{2^k}^+(x_0)} W,\;\;\forall\;k\in \mathbb{N}.
    \end{equation*}
    On the other hand, the oscillation of $W$ over any subset of $\Sigma^\infty$ is bounded by $2\|W\|_{L^\infty(\Sigma^\infty)}$, this combined with the above inequality implies that
    \begin{equation*}
         \osc \limits_{B_{1}^+(x_0)} W  \leq 2 \mu^k\|W\|_{L^\infty(\Sigma^\infty)} \;\;\forall\;k\in \mathbb{N}.
    \end{equation*}
    Since $\mu \in (0,1)$ and $k$ is arbitrary this shows that the oscillation of $W$ in $B_{1}^+(x_0)$ is zero, and since $x_0 \in \Sigma_0$ is arbitrary it follows that $W\mid_{\Sigma_0}$, and thus $W$, must be a constant. 
\end{proof}

%%%%%%%%%%%%%%%%%%%%%%%%%%%%%%%%%%%%%%%%%%%%%%
%%%%%%%%%%%%%%%%%%%%%%%%%%%%%%%%%%%%%%%%%%%%%%

\bibliography{../refs}
\bibliographystyle{plain}
%%%%%%%%%%%%%%%%%%%%%%%%%%%%%%%%%%%%%%%%%%%%%%
\end{document}